\documentclass{article}
\usepackage{amsmath}
\usepackage{amsfonts}
\usepackage{bbold}
\usepackage[mathscr]{eucal}
\usepackage{amsthm}
\usepackage{amssymb}
\usepackage{hyperref}
\usepackage{cleveref}
\usepackage[margin=1.3in]{geometry}
\usepackage[dvipsnames]{xcolor}
\usepackage{graphicx}
\usepackage{mathtools}
\usepackage{multirow}
\usepackage[shortlabels]{enumitem}
\usepackage{import}
\usepackage{tikz-cd}

\usepackage{comment}

\newcounter{COMMENTS}

\makeatletter
\def\firsttwoletters#1#2\@nil{%
  \uppercase{#1}%
  \ifx\relax#2\relax
  \else
    \uppercase{\expandafter\@firstoftwo\expandafter{\expandafter\@car\string#2\@nil}}%
  \fi
}

\newcommand{\newComment}[3]{%
  \expandafter\newcommand\csname #1\endcsname[1]{
    \textbf{%
      \color{#3}(\firsttwoletters#2\@nil\theCOMMENTS)%
    }%
    \marginpar{\scriptsize\raggedright\textbf{%
      {\color{#3}(\expandafter\uppercase\expandafter{\firsttwoletters#2\@nil}\theCOMMENTS)#1: 
      }} ##1}%
    \stepcounter{COMMENTS}%
  }%
  \expandafter\newcommand\csname #2\endcsname[1]{{\color{#3} ##1}}
}
\makeatother

\newComment{Beth}{bb}{Plum}
\newComment{Robbie}{rl}{cyan}
\newComment{George}{gd}{orange}
\newComment{Hannah}{hh}{WildStrawberry}

\usepackage{tcolorbox} 
\newlist{todolist}{itemize}{2}
\setlist[todolist]{label=$\square$}
\usepackage{pifont}
\newtheorem{thm}{Theorem}
\newtheorem{lem}[thm]{Lemma}
\newtheorem{cor}[thm]{Corollary}
\newtheorem{prop}[thm]{Proposition}

\theoremstyle{definition}
\newtheorem{DEF}[thm]{Definition}
\newtheorem{remark}[thm]{Remark}

\newtheorem{innercustomthm}{Theorem}
\newenvironment{customthm}[1]
  {\renewcommand\theinnercustomthm{#1}\innercustomthm}
  {\endinnercustomthm}

\DeclareMathOperator{\homeo}{Homeo}
\DeclareMathOperator{\aut}{Aut}
\DeclareMathOperator{\Homeo}{Homeo}

\DeclareMathOperator{\Stab}{Stab}

\newcommand{\Z}{\mathbb{Z}}
\newcommand{\G}{\Gamma}
\newcommand{\N}{\mathbb{N}}

\newcommand{\cP}{\mathcal{P}}
\newcommand{\cQ}{\mathcal{Q}}
\newcommand{\cR}{\mathcal{R}}

\newcommand{\fP}{\mathfrak{P}}

\newcommand{\defeq}{:=}
\newcommand{\Hoan}{\Homeo(X_{\alpha,n})}

\begin{document}
\title{Graphical models for topological groups: A case study on countable Stone spaces}
\author{Beth Branman, George Domat, Hannah Hoganson and Robert Alonzo Lyman}


\maketitle

\begin{abstract}
  By analogy with the Cayley graph of a group
  with respect to a finite generating set
  or the Cayley--Abels graph of a totally disconnected, locally compact group,
  we detail countable connected graphs
  associated to Polish groups
  that we term Cayley--Abels--Rosendal graphs.
  A group admitting a Cayley--Abels--Rosendal graph
  acts on it continuously, coarsely metrically properly
  and cocompactly by isometries of the path metric.
  By an expansion of the Milnor--Schwarz lemma,
  it follows that the group is generated by a coarsely bounded set
  and the group equipped with a word metric
  with respect to a coarsely bounded generating set
  and the graph are quasi-isometric.
  In other words, groups admitting Cayley--Abels--Rosendal graphs
  are topological analogues of finitely generated groups. Our goal is to introduce this topological perspective on the work of Rosendal to a geometric group theorist.

  We apply these concepts to homeomorphism groups of countable Stone spaces.
  We completely characterize when these homeomorphism groups
  are coarsely bounded, when they are locally bounded (all of them are),
  and when they admit a Cayley--Abels--Rosendal graph,
  and if so produce a coarsely bounded generating set.
\end{abstract}

\section{Introduction}

A \emph{Cayley--Abels graph} for a totally disconnected, locally compact group $G$
is a connected, locally finite graph $\Gamma$
on which $G$ acts continuously, vertex transitively, and with compact stabilizers.
Such a group $G$ is then compactly generated
and in fact quasi-isometric to $\Gamma$
when equipped with a word metric with respect to a compact generating set.
(This metric makes $G$ discrete,
but one can exchange it with a quasi-isometric metric
additionally generating the topology on $G$.)
Cayley--Abels graphs are the extension
of the notion of a Cayley graph (with respect to a finite generating set)
to locally compact groups.

We expand the notion further to the setting of Polish groups,
introducing the notion of a \emph{Cayley--Abels--Rosendal graph.}
Briefly for experts,
a Cayley--Abels--Rosendal graph for a group $G$
is a connected, countable graph $\Gamma$
on which $G$ acts continuously,
vertex transitively,
with coarsely bounded stabilizers
and finitely many orbits of edges.
The group $G$ is then generated by a coarsely bounded set,
and with respect to the word metric associated to such a generating set
becomes quasi-isometric to $\Gamma$.
Further, the word metrics for any two coarsely bounded generating sets are quasi-isometric, giving $G$ a well-defined quasi-isometry type. Just as Cayley graphs are a powerful way to study finitely generated groups geometrically,
so too do we believe that Cayley--Abels--Rosendal graphs
are fundamental to the large-scale geometric study of Polish groups.

Rosendal~\cite{Rosendal},
inspired by work of Roe~\cite{Roe} expanded the framework
of coarse geometry from the realm of discrete or locally compact groups
to the wider setting of Polish topological groups. It is also our intention in this paper to make the fundamentals of this theory
more accessible to the mathematician
familiar with the basic techniques of geometric group theory.

Here are the salient features of this work, intended for a geometric group theorist.
First: a \emph{Polish} topological space is one that is \emph{separable,}
having a countable dense set,
and \emph{completely metrizable}---that is, not just admitting a metric
generating the topology,
but one that is additionally complete.

Now, there are \emph{a priori} many topologies one can put on a group $G$.
As geometric group theorists, we are interested in the \emph{metrizable}
topologies,
since $G$ should be able to act on itself by isometries.
In fact, this latter observation suggests that we are really interested
in \emph{left-invariant} (or right-invariant) \emph{compatible metrics}
on $G$;
that is, those metrics that generate the given topology,
which is then \emph{a fortiori} metrizable,
and that are compatible with the group structure.
As it turns out, by the Birkhoff--Kakutani metrization theorem~\cite{Birkhoff, Kakutani},
every metrizable topological group admits a compatible left-invariant metric,
and moreover a topological group is metrizable if and only if it is Hausdorff
and first countable (i.e.\ every point, or equivalently just the identity,
has a countable neighborhood basis).

Even allowing that (metrizable) topological groups are interesting,
it is not immediately clear from a geometric group theory point of view why
Polish groups might be the ``right'' objects to consider.
Indeed, an abstract group $G$ may admit many metrizable topologies.
Let us offer this suggestion:
the \emph{discrete} topology on a group $G$ is Polish only when
the required countable dense subset of $G$ is $G$ itself---so for Polish groups,
the discrete groups are countable.
As it turns out, for Polish groups,
it is also true that the countable groups are discrete:
a topological group $G$ has an isolated point
if and only if the identity is isolated if and only if it is discrete,
but a nonempty complete metric space with no isolated points is uncountable.

Geometric group theorists remember that although one can define a ``word metric'' on a discrete group $G$
as soon as one has a generating set for $G$,
in order for two word metrics on a group $G$
to yield quasi-isometric metric spaces,
one needs in general for these generating sets to be \emph{finite.}
The reader familiar with locally compact groups may realize that \emph{finite}
may profitably be replaced with \emph{compact,}
but may wonder what one does without local compactness.

The above framing of the ``fundamental observation'' of geometric group theory---i.e. the Milnor--Schwarz Lemma---actually sells it short:
if one has a (metrically) properly discontinuous, cobounded action of a group $G$ on a geodesic metric space $X$,
one concludes that $G$ is \emph{in fact} finitely generated---with care,
one can even extract a finite generating set---and $G$
equipped with \emph{any} word metric with respect to some finite generating set
is quasi-isometric to the space $X$.
(For finitely generated groups, one can always consider for $X$ a Cayley graph for $G$ with respect to some finite generating set. This space $X$ is then additionally \emph{proper,} that is, closed balls are compact,
so one can replace ``metrically properly discontinuous'' with ``properly discontinuous'' and ``cobounded'' with ``cocompact'' with no loss.)

As it turns out, the correct enlargement of finiteness or compactness is the notion of \emph{coarse boundedness,} and that with this expanded notion, one still has a Milnor--Schwarz Lemma.
This allows one to compute quasi-isometry types of those Polish groups that admit them.
Moreover, the only coarsely bounded subsets of a discrete group are the finite ones,
so this expanded field is really a conservative extension of the theory:
for countable discrete groups,
it continues to pick out the finitely generated ones.

To illustrate their use,
we construct Cayley--Abels--Rosendal graphs
for a family of homeomorphism groups.
A \emph{Stone space} is compact, Hausdorff, and totally disconnected.
Stone spaces are dual to Boolean algebras,
and second countable Stone spaces show up as spaces of ends of
surfaces and locally finite graphs.
The group of homeomorphisms of a Stone space
is a non-Archimedean Polish group,
and may profitably be thought of as a ``$0$-dimensional'' analogue
of the mapping class group of a locally finite infinite graph or surface of infinite type.

Countable Stone spaces are classified up to homeomorphism by two parameters:
a positive integer $n$ and a countable ordinal $\alpha$. Here the ordinal $\alpha$ refers to the Cantor-Bendixson rank of the space and $n$ denotes the number of points with maximal Cantor-Bendixson rank. We will write $X_{\alpha,n}$ to denote the corresponding countable Stone space. For more details, see \Cref{ssec:ctblestonespaces}. Recall that ordinals are either \emph{successor} ordinals,
being of the form $\alpha = \beta + 1$,
or they are not, in which case they are \emph{limit} ordinals.

\begin{customthm}{A}
  \label{thm:mainclassification}
  Let $X_{\alpha,n}$ be a countable Stone space with $\alpha>0$.
  The group $\Homeo(X_{\alpha,n})$
  is always locally bounded and is
  \begin{itemize}
  \item coarsely bounded if and only if $n = 1$, and
  \item boundedly generated but not coarsely bounded if and only if $n > 1$ and $\alpha$ is a successor ordinal.
  \end{itemize}
  Moreover, if $\Homeo(X_{\alpha,n})$ is boundedly generated,
  we compute its quasi-isometry type by constructing
  a Cayley--Abels--Rosendal graph for it.
\end{customthm}

When $\alpha=0$, $X_{\alpha,n}$ is a finite discrete set, so its homeomorphism group is coarsely bounded.

Some parts of \Cref{thm:mainclassification} follow from the work of Mann--Rafi \cite[Theorem 1.5]{MR2023},
and other parts provide new proofs of some of their work.
Just as every quotient of a finitely generated group is finitely generated,
so too is every continuous quotient of a boundedly generated group boundedly generated, see \Cref{lem:countgen}.
Similarly, every continuous quotient of a coarsely bounded group is coarsely bounded. The mapping class group of a surface acts continuously on its space of ends, which is a Stone space.
Since Mann--Rafi prove that the mapping class group of a genus-zero surface with end space homeomorphic to $X_{\alpha,n}$ is boundedly generated when $\alpha$ is a successor ordinal 
and coarsely bounded when $n = 1$ with no stipulation on $\alpha$,
these parts of \Cref{thm:mainclassification} follow from \cite[Theorem 1.5]{MR2023}.
On the other hand, since the properties of being \emph{not} boundedly generated and locally bounded are easily seen to be inherited from continuous quotients,
these parts of \Cref{thm:mainclassification} imply the corresponding statements for the mapping class group of \emph{every} surface with appropriate countable end space.
In all cases, the construction of Cayley--Abels--Rosendal graphs is new.

We end this introduction with a brief survey and list of examples to which this framework of Rosendal applies. First, as previously mentioned, this framework subsumes the setting of locally compact groups, e.g. see \cite{CdH2016,CCMT2015,Lederle2022} among many others. For those working in low-dimensional topology, a natural Polish (but not locally compact) group is the homeomorphism group of a manifold. These groups were seen to admit a well-defined quasi-isometry type when the surface is compact in \cite{MR2018}, (see also \cite{Vlamis2024} for some non-compact surface examples). Related to these groups we have the mapping class group of a surface: The quotient of the homeomorphism group by the connected component of the identity. In the finite-type setting, the quotient topology is discrete and these groups are finitely generated. However, when the underlying surface is of infinite type (does not have finitely generated fundamental group), the mapping class group is not discrete nor locally compact, but is still Polish. These \emph{big} mapping class groups were studied using this framework in \cite{MR2023}, see also \cite{RM2023,HKR2023,SC2024}. Along the same lines, the coarse geometry of ``big" analogs of $\operatorname{Out}(F_{n})$ and mapping class groups of related three manifolds were studied in \cite{DHK2022,DHK2023, DHK2023_2,Udall2024}. 

Beyond the world of low-dimensional topology, isometry groups of infinite dimensional hyperbolic spaces and separable Hilbert spaces \cite{Duchesne2023}, general Banach-Lie groups \cite{ADM2022}, and groups of measure-preserving transformations \cite{LeMaitre2021} are all Polish groups. The theory applies to automorphism groups of countable simplicial complexes \cite[Chapter 6]{Rosendal}. In fact, the isometry group of \emph{any} complete, separable metric space can be given the structure of a Polish group \cite[Section 9.B Example 9]{Kechris1995} and so can be studied via this framework. 

We do remark that the general construction of Cayley-Abels-Rosendal graphs in \Cref{prop:CARgraph} may not apply to all of these examples as it requires the existence of an \emph{open}, coarsely bounded subgroup. The construction is readily applicable to \emph{non-Archimedean} Polish groups, which are those with a neighborhood basis of the identity given by open subgroups. Non-Archimedean Polish groups turn out to be exactly the closed subgroups of the symmetric group on a countable set. Among the examples above, these include: Big mapping class groups, big $\operatorname{Out}(F_{n})$ analogs, and automorphism groups of countable simplicial complexes, as well as homeomorphism groups of Stone spaces. 
That being said, the general approach that we take in \Cref{sec:HomeoCtbleStone} of constructing a graph or space and using the Milnor-Schwarz Lemma (\Cref{milnorschwarz}) to build a quasi-isometry between the space and the group in question is applicable to any Polish group. As such, we hope that our exposition is beneficial to anyone attempting to study the coarse geometry of a Polish group, even when the group is not non-Archimedean.


\subsection*{Acknowledgments} 
Thank you to Christian Rosendal and Carolyn Abbott for pointing out two omissions in the first version of this paper. Thank you to the referee for helpful comments that have improved the exposition. GD was partially supported by NSF DMS-2303262. BB was partially supported by NSF DMS-1839968.  HH was supported by NSF DMS-2303365. RAL was partially supported by NSF DMS-2202942.

\section{Coarsely bounded subsets and the Milnor--Schwarz Lemma}
A subset $A$ of a topological group $G$ is said to be \emph{coarsely bounded in $G$}
if whenever $G$ acts continuously by isometries on a (pseudo-)metric space $X$,
some, and hence every, orbit of $A$ is bounded in $X$. 

Let us remark that if $A \subset G$ is finite (or precompact),
then it is coarsely bounded.
The converse is also true for discrete (or locally compact) groups.
If $G$ is finitely generated, this follows by examining the action of an infinite subset $A$ on some Cayley graph of $G$ with respect to a finite generating set. Although the definition above is conceptually clear,
it is functionally difficult to verify.
Fortunately, Rosendal gives us the following formulation.
We give a proof for completeness.

\begin{lem}[cf. Proposition 2.15 of~\cite{Rosendal}]\label{lem:RosendalsCriterion}
    Suppose that $G$ is metrizable and that for every identity neighborhood $U$,
  there exists a countable subset $C$ of $G$
  such that $G$ is generated by $U \cup C$.
  (This holds for $G$ Polish, for instance.)
  
  A subset $A$ of $G$ is coarsely bounded in $G$
  if and only if for every identity neighborhood $U \subset G$,
  there exists a finite subset $F \subset G$
  and $N \in \mathbb{N}$ such that
  \[
    A \subset {(FU)}^N = \{f_1u_1\ldots f_\ell u_\ell : \ell \le N, f_i \in F\text{ and } u_i \in U \}.
  \]
\end{lem}

In fact, we will show that both coarse boundedness and Rosendal's criterion
are equivalent, for $G$ as in the statement,
to the following property:

For every nested open sequence $V_1 \subset V_2 \subset \cdots$ that exhausts $G$
and such that $V_nV_n \subset V_{n+1}$ for each $n$,
we have that $A \subset V_n$ for some $n$.

As we will see shortly, coarse boundedness is readily implied by this property,
but showing that coarse boundedness implies it requires \emph{constructing} (pseudo-)metrics adapted to a given sequence.
To do so in the proof below,
we rely on the following variant of the lemma that Birkhoff~\cite{Birkhoff} uses in his proof of the Birkhoff--Kakutani metrization theorem (see also~\cite{Kakutani}).

\begin{lem}[Lemma 2.13 of~\cite{Rosendal}]\label{lem:Rosendal213}
Suppose $\cdots \subset U_{-1} \subset U_0 \subset U_1 \subset \cdots$ 
is a sequence of symmetric open identity neighborhoods exhausting $G$
and satisfying that $U_n^3 \subset U_{n+1}$ for all $n$.

If we write
\[
\delta(f,g) = \inf\{ 2^n : g^{-1}f \in U_n \}
\]
and
\[
d(f,g) = \inf\left\{ \sum_{i = 1}^n \delta(h_{i-1}, h_i) : f = h_0, h_1, \ldots, h_n = g \in G \right\},
\]
then $d$ is a continuous, left-invariant pseudometric on $G$
satisfying $\frac{1}{2}\delta(f,g) \le d(f, g) \le \delta(f, g)$ for all $f$, $g \in G$.
It is a metric (i.e. separates points)
when the $U_i$ form a neighborhood basis for the identity.
\end{lem}

\begin{proof}[Sketch of proof]
We refer the reader to Birkhoff~\cite{Birkhoff} for more complete details.
That $d$ is a continuous, left-invariant pseudometric on $G$ is clear:
nonnegativity, continuity and symmetry of $d$ follow from these facts for $\delta$,
which in turn follow from the fact that each $U_n$ is symmetric and open,
and the triangle inequality for $d$ follows by realizing one may 
concatenate the sequences $h_i$.
Similarly the inequality $d(f, g) \le \delta(f, g)$ is clear,
as is the fact that $\delta$ separates points when the $U_n$
form a neighborhood basis.

It is the reverse inequality which uses the assumption that $U_n^3 \subset U_{n+1}$:
If $f = h_0,h_1,\ldots,h_n = g$ is a sequence,
we may argue inductively that there is a $k$ satisfying $1 \le k \le n$
such that $\delta(f, h_{k-1})$, $\delta(h_{k - 1}, h_k)$ and $\delta(h_k, g)$
are all at most equal to $\epsilon = \sum_{i = 1}^n \delta(h_{i - 1}, h_i)$.
But from the containment $U_n^3 \subset U_{n+1}$ and the definition of $\delta$,
it follows that $\delta(f, g)$ can be at most twice $\epsilon$.
\end{proof}

\begin{proof}[Proof of \Cref{lem:RosendalsCriterion}]
As mentioned before the proof,
we show the equivalence of both coarse boundedness and Rosendal's criterion to
the following property:

For every nested open sequence $V_1 \subset V_2 \subset \cdots$ that exhausts $G$
and such that $V_nV_n \subset V_{n+1}$ for each $n$,
we have that $A \subset V_n$ for some $n$.

Indeed, if $A$ satisfies the above property
and $(X,d)$ is a metric space with a continuous $G$-action,
let $x_0 \in X$ be a basepoint
and consider the sets $V_n = \{ g \in G : d(x_0, g.x_0) < 2^n \}$.
These are nested, open and satisfy $V_n V_n \subset V_{n+1}$,
whence the $A$-orbit of $x_0$ is $d$-bounded.
Since $(X,d)$ and $x_0$ were arbitrary, we conclude that $A$ is coarsely bounded.

An $A$ satisfying the above property also satisfies Rosendal's criterion:
given $U$ an open identity neighborhood in $G$,
there exists a countable set $\{x_0, x_1, \ldots \}$ that,
together with $U$, generates $G$.
Write $F_n = \{x_0,\ldots, x_n\}$.
Consider the sets $V_n = (F_nU)^{2^n}$.
The sets $V_n$ are constructed for us to apply the above property;
one such $V_n$ contains $A$, proving Rosendal's criterion.

Next, suppose $A$ satisfies Rosendal's criterion and consider a nested sequence $V_1 \subset V_2 \subset \cdots$.
As $V_k$ is an identity neighborhood for some $k$,
we may take $F$ and $n$ so that $A \subset (FV_k)^n$.
Since the sets $V_i$ exhaust, $F$ is contained in some $V_m$ for $m \ge k$,
and repeatedly applying the assumption that $V_m V_m \subset V_{m+1}$
allows us to conclude that in fact $A \subset V_{m + n}$.

Finally, suppose that $A$ is coarsely bounded
and take a nested sequence as in the statement of the property.
We transform it to a new sequence to which we may apply \Cref{lem:Rosendal213}.
If the $V_i$ are not symmetric, we may consider $U_i = V_i \cap V_i^{-1}$.
These sets are symmetric, open, exhaust $G$
and still satisfy $U_n U_n \subset U_{n+1}$.
Ignoring finitely many of the $U_i$ we may assume $1 \in U_1$,
and by passing to an appropriate subsequence
we may upgrade to the property that $U_n^3 \subset U_{n+1}$.
We may complete the sequence $U_i$ in the negative direction
so that the $U_i$ form a neighborhood basis for the identity.
(Indeed, Birkhoff~\cite{Birkhoff} proves that any first-countable Hausdorff topological group
has a countable neighborhood basis $\{U_{-1},U_{-2},\ldots\}$
of the identity with the property that $U_n^3 \subset U_{n+1}$.)

This done, \Cref{lem:Rosendal213} 
produces a continuous left-invariant metric $d$ on $G$ for which
a set is $d$-bounded 
(i.e. has bounded orbits under the action of $G$ on itself by isometries of $d$) 
if and only if it appears in some $U_n$.
Since the $U_n$ are derived from the $V_n$ by symmetrizing,
passing to a subsequence, and then augmenting in the negative direction,
each $d$-bounded set will be contained in some $V_n$.
Therefore, since $A$ is coarsely bounded,
we see that it satisfies the desired property.
\end{proof}

Suppose that $G$ acts continuously by isometries on a metric space $X$.
The space $X$ is \emph{geodesic} if between any two points of $X$
there exists
a rectifiable curve from one to the other whose length
is equal to the distance between them---such a curve is a \emph{geodesic.} $G$ is \emph{locally bounded} (some authors have ``locally CB'')
if it has a coarsely bounded identity neighborhood
and \emph{boundedly generated} (some authors have ``CB generated'')
if it is generated by a coarsely bounded set. Let us remark that \emph{a priori} bounded generation appears to be
somewhat orthogonal to local boundedness.
This is not so for Polish groups, as we now show.

\begin{lem}\label{boundedlygeneratedimplieslocallybounded}
  Suppose that $G$ is a boundedly generated Polish group.
  Then $G$ is locally bounded.
\end{lem}

First, some useful terminology.
A subset $S$ of a space $X$ is \emph{nowhere dense}
if its closure has empty interior.
A countable union of nowhere dense sets is said to be \emph{meagre,}
while a set containing a countable intersection of open and dense sets
is said to be \emph{residual.}
A quick set-theoretical argument shows that the residual sets are exactly
the \emph{comeagre sets,}
i.e.\ those whose complement is meagre.
A set is \emph{non-meagre} if it is not meagre.
The \emph{Baire category theorem} says that in a Polish space,
residual sets are dense.

Observe that if $N$ is nowhere dense, then $X - \bar{N}$ is open and dense.
If some nonempty open set $U$ of a Polish space $X$
(or just a space satisfying the Baire category theorem)
were meagre,
say equal to $\bigcup_{n\in\mathbb{N}} N_n$ where each $N_n$
is nowhere dense,
we would have that the closed set $\overline{\bigcap_{n\in\mathbb{N}}(X - N_n)}$
is equal to the complement of $U$, in contradiction to the fact
that it is residual and hence dense.
Therefore in Baire spaces,
open sets are non-meagre.
It is also clear that a non-meagre closed set must have nonempty interior.

A subset $A \subset X$ has the \emph{Baire property} if there exists
an open set $U \subset X$
such that the symmetric difference $A \triangle U$ is meagre.
Equivalently we may write $A$ as the symmetric difference
$A = U \triangle M$
where $U$ is open and $M$ is meagre.
If $X$ is Polish, we say that $A \subset X$ is \emph{analytic}
if it is the continuous image of some Polish space.
Borel sets are analytic,
and analytic subsets have the Baire property.

\begin{lem}[Pettis' Lemma, Theorem 1 of~\cite{Pettis}]
  Suppose that $A$ is a non-meagre analytic subset
  of a Polish topological group $G$.
  Then $A^{-1}A$ is an identity neighborhood.
\end{lem}

\begin{proof}
  Since $A$ has the property of Baire,
  we may write it as $W \triangle M$
  for some open set $W$ and some meagre set $M$.
  By continuity of the map $(g,h) \mapsto gh^{-1}$,
  every open neighborhood $V$ of the identity contains an open neighborhood
  $U$ such that $UU^{-1} \subset V$.
  Since arbitrary open sets are left-translates of open identity neighborhoods,
  we see that there exists $g \in G$ and an open identity neighborhood $U$
  such that $gUU^{-1} \subset W$.
  We claim that $U \subset A^{-1}A$.
  Indeed, take $h \in U$.
  Observe that $E_h = (G - M) \cap [(G - M)h^{-1}]$
  is residual, being the intersection of two comeagre sets,
  so has nontrivial intersection with the open set $gUh^{-1}$.
  Take $x$ in this intersection.
  We claim that $x$ and $xh$ are in $A$,
  whence $h \in A^{-1}$.
  Indeed, by the definition of $E_h$,
  both are in the complement of $M$,
  while by the containment $gUh^{-1} \subset gUU^{-1}$,
  we have that they are contained in $W$.
\end{proof}

\begin{proof}[Proof of \Cref{boundedlygeneratedimplieslocallybounded}]
  Let $S$ be a coarsely bounded set generating $G$;
  we may assume that it is symmetric without loss.
  To say that $S$ generates is equivalent to the observation that the union
  of the sequence $S, S^2, \ldots$ is all of $G$,
  hence the same is true of the sequence $\overline{S}, \overline{S^2}, \ldots$.
  Observe that each $S^k$ is coarsely bounded by Rosendal's criterion.
  By the definition of coarse boundedness, so are their closures,
  which have the property of Baire, being Borel.
  Since $G$ is Polish it is non-meagre,
  so some $\overline{S^k}$ must be non-meagre.
  Then $S^{2k} = {(S^k)}^{-1}S^k$ is a coarsely bounded identity neighborhood
  by Pettis's lemma.
\end{proof}

We come now to the Milnor--Schwarz Lemma.
To state it, we need a couple definitions.
Recall that in the setting of a topological group $G$ acting (continuously) on a topological space $X$,
the action is \emph{proper} if the map $G\times X \to X\times X$
given by $(g,x) \mapsto (g.x, x)$ is a proper map;
that is, preimages of compact sets are compact.
When $X$ is locally compact and Hausdorff,
this is equivalent to the condition that for each compact set $K \subset X$, the set
\[
  \{ g \in G : g.K \cap K \ne \varnothing \}
\]
is compact in $G$.
When $G$ is discrete,
the compact sets are finite
and this latter condition is ``proper discontinuity'' of the action.
When $X$ is a (proper, i.e.~closed metric balls are compact) metric space,
we may instead consider closed metric balls $B$ instead of compact sets and obtain the notion
of a \emph{metrically proper} action.
Without the assumption that $X$ is \emph{proper,} these conditions are really different.
If instead of compact, we require that the sets
\[
  \{ g \in G : g.B \cap B \ne \varnothing \}
\]
are coarsely bounded,
we have the notion of a \emph{coarsely metrically proper} group action. A group action is \emph{cocompact} (respectively \emph{cobounded})
if there is a compact (respectively bounded) set whose translates cover the space.

\begin{prop}[Milnor--Schwarz Lemma]\label{milnorschwarz}
  Suppose that the Polish group $G$ acts continuously,
  coarsely (metrically) properly
  and coboundedly by isometries
  on a geodesic metric space $(X,d)$.
  Then $G$ is boundedly generated,
  and although the word metric with respect to a symmetric, analytic,
  coarsely bounded generating set
  need not be compatible with the topology on $G$
  (in fact, it willl not be if $G$ is not discrete),
  any such metric is quasi-isometric---via any orbit map---to $(X,d)$.
\end{prop}

Let us remark that in the proof we will need to use a metric
which is not compatible with the topology on $G$.
We will show that it is quasi-isometric to a metric
that \emph{is} compatible in \Cref{wordmetricsQItoactualmetrics}
directly following the proof.

\begin{proof}
  The proof is essentially the classical geometric group theory proof.
  Afficionados of the proof will recognize that we do not need the full strength of ``geodesic.''
  We also do not need that $G$ is Polish except in proving \Cref{wordmetricsQItoactualmetrics}.

  Let $B$ be an open bounded set whose $G$-translates cover $X$,
  (say a sufficiently large metric ball)
  and consider the set $S \subset G$
  of elements $g \in G$ such that $g.B \cap B \ne \varnothing$.
  Let us remark that by continuity this set is open:
  the sets $\{ (g,x) : g.x \in B \}$ and $\{ (g, x) : x \in B\}$ are open,
  and $S$ is the image of their intersection under the (open)
  projection map from $G \times X$ to $G$.
  By coarse properness of the action, the set $S$ is coarsely bounded.

  Consider $g \in G$ and $x_0 \in B$.
  Let $\gamma$ be a geodesic from $x_0$ to $g.x_0$.
  Since translates of $B$ cover $X$,
  they certainly cover the image of $\gamma$.
  By compactness of this image, finitely many suffice,
  say $B,g_1.B\ldots,g_k.B, g.B$,
  where we have ordered the $g_i$ so that $g_i.B \cap g_{i+1}.B \ne \varnothing$,
  and we write $1 = g_0$ and $g = g_{k+1}$.
  By writing $s_i = g_i^{-1}g_{i+1}$,
  we see that each $s_i$ for $0 \le i \le k$ belongs to $S$
  and that $g = s_0\cdots s_k$.
  Therefore the set $S$ is an open coarsely bounded generating set for $G$.

  Next suppose that $S$ is a symmetric, \emph{analytic}
  coarsely bounded generating set for $G$,
  and consider the word metric on $G$ with respect to $S$.
  Analyticity is necessary to prove in \Cref{wordmetricsQItoactualmetrics}
  below that this word metric is quasi-isometric to one with respect to an open,
  coarsely bounded generating set.
  Since $S$ is coarsely bounded,
  we have that for any $x_0 \in X$,
  there exists $M$ large such that
  for each $g \in S$, we have that $d(x_0, g.x_0) \le M$.
  By the triangle inequality, if $g \in G$ is arbitrary,
  we have that $d(x_0, g.x_0)$ is bounded above by $M$
  times the word length of $g$ with respect to $S$.
  Conversely, suppose that $B$ is an open ball centered at $x_0$
  of radius $2C$,
  where translates of the ball of radius $C$ cover $X$.
  The collection $S'$ of elements of $G$ which fail to move $B$
  off of itself
  is coarsely bounded in $G$,
  hence by \Cref{wordmetricsQItoactualmetrics} to follow,
  we see that word lengths of elements in $S'$ with respect to $S$
  are bounded by some constant $M'$.
  Write
  \[
    k = \left\lfloor \frac{d(x_0,g.x_0)}{C} \right\rfloor.
  \]
  We can find a sequence of $k+1$ points $x_0,x_1,\ldots,x_k = g.x_0$
  on any geodesic from $x_0$ to $g.x_0$ such that the distance between
  successive points is at most $C$.
  Associated to these points we can find $g_i \in G$
  such that $x_i \in g_i.B$,
  from which we conclude that the word length of $g$
  with respect to $S$ is at most
  \[
    Mk + M \le M \frac{d(x_0, g.x_0)}{C} + M,
  \]
  showing that the word metric on $G$ with respect to $S$
  is quasi-isometric to the metric
  on the orbit of $x_0$,
  or equivalently by coboundedness, the metric on $X$.
\end{proof}

To complete the proof, we need the following lemma.
Analyticity is used in the proof in order to appeal to Pettis's lemma.

\begin{lem}[Lemma 2.51 of~\cite{Rosendal}]\label{wordmetricsQItoactualmetrics}
  Suppose that $d$ is a word metric on a Polish group $G$
  with respect to some symmetric, analytic coarsely bounded generating set $S$.
  There exists a left-invariant compatible metric $\partial$ on $G$
  such that $(G,\partial)$ is quasi-isometric to $(G,d)$.
\end{lem}

\begin{proof}
  First, observe that we may suppose that $S$ is \emph{open}.
  Indeed, if it is not already,
  observe that because $G$ is Polish and boundedly generated,
  it is locally bounded by \Cref{boundedlygeneratedimplieslocallybounded},
  so we may take some coarsely bounded, symmetric open identity neighborhood $U$.
  By Rosendal's criterion,
  there exists a finite (we may assume symmetric) set $F$
  such that $S \subset {(FU)}^k$.
  But this says precisely that $S$ has bounded word lengths with respect
  to the open generating set $FU$.
  Conversely, $FU$ is coarsely bounded,
  so if we can show that some $S^k$ contains an identity neighborhood,
  we can conclude that $FU$ has bounded word lengths with respect to $S$,
  and thus that the corresponding word lengths are bi-Lipschitz equivalent.

  This is where we need analyticity of $S$:
  the point is that we yet again have $G = \bigcup_{k \in \mathbb{N}} S^k$,
  so some $S^k$ is nonmeagre and has the property of Baire,
  so by Pettis's lemma, since $S^k$ is symmetric,
  $S^{2k}$ is an identity neighobrhood.
  Using an open identity neighborhood $V \subset S^{2k}$
  and Rosendal's criterion for coarse boundedness of $FU$,
  we see that $FU$ has bounded word lengths with respect to $S$,
  so their word metrics are bi-Lipschitz equivalent.

  Supposing that $S$ is an open identity neighborhood,
  which is a symmetric and coarsely bounded generating set,
  write $\|g\|$ for the word norm on $g$ with respect to $S$.
  Let $d'$ be \emph{any} compatible left-invariant metric on $G$
  with corresponding norm $|g| = d'(1,g)$.
  Because $S$ is coarsely bounded,
  we have some constant $M$ such that
  for each $h \in S$, we have $|h| \le M$.
  Define a metric $\partial$ on $G$ by the rule that
  \[
    \partial(g,h) = \inf \{ |s_1| + \cdots + |s_k| : h^{-1}g = s_1\ldots s_k, \ s_i \in S \}.
  \]
  First, it is not hard to argue that this really is a left-invariant metric
  on $G$, since $d'$ is one.
  Since $\partial$ agrees with $d'$ on the open set $S$
  by the triangle inequality
  and otherwise satisfies $\partial \ge d'$,
  we see that $\partial$ is compatible with the topology on $G$.
  Notice that we have $\partial(g,h) \le M\|h^{-1}g\|$ by construction.
  On the other hand, because $S$ is open, it contains the $d'$
  ball of some radius $2\epsilon > 0$ about the identity.
  Since $\partial(g,h)$ is defined as an infimum,
  we may take some actual sequence $s_1,\ldots,s_k$ from $S$
  such that $h^{-1}g = s_1\ldots s_k$
  and such that $|s_1| + \cdots + |s_k| \le \partial(g,h) + 1$.
  Indeed, among such sequences, choose $k$ to be the smallest possible.
  Now, notice that if $s_i s_{i+1}$ were in $S$,
  we could do better on choosing $k$,
  so in fact they do not belong to $S$.
  It follows that $|s_i s_{i+1}| \ge 2\epsilon$,
  whence by the triangle inequality
  at least one of $|s_i|$ or $|s_{i+1}|$ must be at least $\epsilon$.
  Now, we have
  $1 + \partial(g,h) \ge |s_1| + \cdots + |s_k| \ge \lfloor \frac{k}{2} \rfloor \cdot \epsilon$,
  and the word length $\|g^{-1}h\| \le k$ by assumption,
  so
  \[
    \frac{\epsilon\|g^{-1}h\|}{2} - 1 - \frac{\epsilon}{2} \le \partial(g,h),
  \]
  from which we conclude that the identity map from $(G,d)$ to $(G,\partial)$
  is a quasi-isometry.
\end{proof}

By the by,
there is also a ``finite generation'' condition underlying
the notion of bounded generation,
as we show in the next lemma.
Recall that for a countable discrete group $G$,
being finitely generated is equivalent to the condition
that whenever $G$ may be written as the union of a chain of subgroups
$G_1 \le G_2 \le \cdots$,
this chain terminates in $G$ after finitely many steps.

\begin{lem}[cf. Theorem 2.40 of~\cite{Rosendal}]\label{lem:countgen}
  Suppose that $G$ is locally bounded and Polish.
  Then $G$ is boundedly generated if and only if
  every ascending chain of open subgroups
  that exhausts the group
  terminates after finitely many steps.
\end{lem}

\begin{proof}
  If $G$ is boundedly generated, say by $S$,
  let $U = G_1$ be our first open subgroup.
  By Rosendal's criterion,
  there exists a finite set $F$ and $k$ such that $S \subset {(FU)}^k$.
  But $F$, and hence ${(FU)}^k$ and $G$,
  is therefore contained in some $G_n$,
  so the chain terminates in $G$ after finitely many steps.

  Conversely, since $G$ is locally bounded and Polish,
  there exists an open, coarsely bounded set $U$
  and a countable collection $\{x_1,x_2,\ldots\}$ of elements
  such that $G = \langle U, x_1,x_2,\ldots \rangle$.
  But then by considering the open subgroups
  $G_n = \langle U, x_1, \ldots, x_n \rangle$,
  we see that finitely many of the $x_i$ suffice.
  Since the set $U \cup \{x_1,\ldots,x_k\}$
  is clearly coarsely bounded for any $k$ by Rosendal's criterion,
  we see that $G$ is boundedly generated.
\end{proof}

\section{Cayley--Abels--Rosendal graphs}\label{sec:CARgraphs}

The results in this section are inspired by
\cite[Section 6.2]{Rosendal} and \cite[Section 2]{Lederle2022}.
There the authors work in the settings of non-Archimedean Polish groups 
(the closed subgroups of the symmetric group on a countable set) 
or totally disconnected, locally compact groups. 
In each case, the groups considered have a wealth of open subgroups, 
making them well-suited to the constructions we describe here. 
Our perspective is broadly very similar.
One difference is that since we work in greater generality,
our results show that their constructions are in some sense
the only ones available.

First, a definition.
A connected, countable simplicial graph $\Gamma$
is a \emph{Cayley--Abels--Rosendal graph} for a topological group $G$
if $G$ admits a continuous,
vertex-transitive action on $\Gamma$
with finitely many orbits of edges
and coarsely bounded (necessarily open) vertex stabilizers.

Note: Our definition follows more closely the definition of
a Cayley--Abels graph for a totally disconnected locally compact group,
rather than that of a Cayley graph.
While groups act vertex-transitively on their Cayley graphs,
as commonly defined, there is no assumption on the finiteness of the generating set,
nor continuity of the action.

Next we turn to some generalities on continuous actions of groups on graphs.
Suppose that a group $G$ acts on a graph $\Gamma$ continuously.
This action induces a representation from $G$ into the group of bijections of the vertex set $V\Gamma$.
This latter group is
a topological group with the \emph{permutation topology,}
where pointwise stabilizers of finite sets are basic open neighborhoods of the identity,
so this representation will be continuous if and only if the stabilizer of a vertex is open in $G$.
Such a permutation actually arises as a graph automorphism
just when it preserves adjacency.

When $\Gamma$ is simplicial, adjacency is a (symmetric) \emph{relation} on $V\Gamma$,
that is, a subset of $V\Gamma \times V\Gamma$,
so a bijection of $V\Gamma$ corresponds (uniquely) to a graph automorphism
just when it preserves this relation in the diagonal action on $V\Gamma \times V\Gamma$.
It is not hard to see that
the topology on $\aut(\Gamma)$ is precisely the subspace topology,
so if a continuous representation from a group $G$ into the group of bijections of $V\Gamma$
preserves the adjacency relation,
then $G$ acts continuously on $\Gamma$,
and we see that this happens if and only if vertex stabilizers are open in $G$.

Now when $\Gamma$ is a simplicial graph and $G$ acts continuously with one orbit of vertices,
write $V$ for the stabilizer of a fixed vertex $v$, and
notice that there exists an element $k \in G$ sending the oriented edge $e = (g.v, h.v)$
to the oriented edge $e' = (g'.v, h'.v)$ precisely when the pair $(kgV, khV)$
is equal to the pair $(g'V, h'V)$,
or put another way, when the \emph{double coset} $Vg^{-1}hV$
is equal to the double coset $Vg'^{-1}h'V$.
So when $\Gamma$ is simplicial,
every orbit of oriented edges corresponds (uniquely)
to a double coset of $V$ in $G$.

Let $U$ be the (open) set of elements $g$ in $G$
such that $g.v$ is either equal to $v$ or adjacent to it.
Observe that $VUV = U$:
if $g \in U$ and $h$, $k \in V$,
first we note that we have $hgk.v = hg.v$.
If $g.v = v$, then $hg.v = v$, so $hgk \in U$.
On the other hand, if $g.v = w$ which is adjacent to $v$,
then since $h.v = v$, the element $h$ must send $w$
to another vertex adjacent to $v$,
whence $hgk \in U$.

Thus if $g \in U$, then actually its whole coset $gV$ 
is contained in $U$.
Let $A$ be a set of representatives for the cosets in $G/V$
contained in $U$.
Since $U$ is symmetric by definition,
we may choose $A$ to be symmetric,
whence $U = VUV = AV = {(VA)}^{-1} = V^{-1}A^{-1} = VA$, since $V$ is a subgroup.
When $G$ is alternately Polish or $G$ is locally compact
and $U$ is compact,
the set $A$ is correspondingly countable or finite,
and is in one-to-one correspondence
with the set of vertices of $\Gamma$
equal to or adjacent to $v$.
Thus when $\Gamma$ is simplicial,
$A - \{1\}$ is in one-to-one correspondence
with the set of edges incident to $v$.

Let $F$ be a set of representatives for the double cosets
corresponding to oriented edge orbits in $\Gamma$.
We choose the notation $F$ to follow Rosendal~\cite{Rosendal}
and note that in the case that $\Gamma$ is a Cayley--Abels--Rosendal graph, $F$ will be a finite set.
We may choose $F$ 
such that $F \subset (A - \{1\})$
and such that $F$ is symmetric.

Conversely, given an open subgroup $V \le G$ and a set of double cosets in $V \backslash G / V$ represented by the set $F$ chosen to be symmetric and meet $V$ trivially,
we construct a simplicial graph $\Gamma = \Gamma(V,F)$.
Its vertex set is $G/V$ and two vertices (cosets) $gV$ and $hV$ are declared adjacent when the double coset $Vg^{-1}hV$ has a representative in $F$.
Put more plainly, we declare that the vertex (coset)
$gV$ is adjacent to $gafV$ for each $a \in V$ and $f \in F$.
The set of edges of the simplicial graph $\Gamma$ incident to the vertex $V$
is thus in bijection with the set of cosets $gV$
contained in the open set $VFV$ 
(notice that this set is likely much smaller than the set of elements $VF$).
The action of $G$ on $\Gamma(V,F)$
is continuous, vertex transitive,
and $F$ is in one-to-one correspondence
with the set of oriented edge orbits.

It is not hard to see that the graph $\Gamma$ will be connected precisely when $VFV \cup V$ (which is equal to $U = AV = VA$ in the discussion above) is an open generating set for $G$.
Notice that any continuous vertex-transitive action of a group $G$
on a connected graph $\Gamma$
becomes cobounded as soon as we associate to $\Gamma$
the geodesic path metric that assigns length $1$ to each edge.
If such an action is to be coarsely metrically proper,
it must be the case that the set
$VFV \cup V$ associated to $\Gamma$ as above
is coarsely bounded in $G$,
since it fails to move the closed $1$-neighborhood of $v$ off of itself.
But since $\Gamma$ is connected,
the coarsely bounded open set $VFV \cup V$ generates $G$.

\begin{prop}\label{prop:CARgraph}
    Suppose that $G$ admits a Cayley--Abels--Rosendal graph.
    Then $G$ is boundedly generated and by the Milnor--Schwarz Lemma is quasi-isometric to any Cayley--Abels--Rosendal graph for $G$.
    If $G$ is boundedly generated and satisfies the hypotheses of \Cref{lem:RosendalsCriterion}, then for any open, coarsely bounded subgroup $V$ of $G$ (supposing one exists),
    there exists a Cayley--Abels--Rosendal graph for $G$ with $V$ as the stabilizer of a vertex.
\end{prop}

\begin{proof}
    The foregoing argument shows that if $G$ has a Cayley--Abels--Rosendal graph $\Gamma$, then the graph $\Gamma$ is $G$-equivariantly isomorphic to $\Gamma(V,F)$
    for some open coarsely bounded subgroup $V$
    and finite set of double coset representatives $F$.
    The set $VFV$ generates $G$ since $\Gamma(V,F)$ is connected,
    as does $V \cup F$, which is coarsely bounded
    since $F$ is finite.

    Supposing that $G$ is boundedly generated and satisfies the hypotheses of \Cref{lem:RosendalsCriterion}
    and that $V$ is an open, coarsely bounded subgroup,
    we may apply Rosendal's Criterion to a coarsely bounded generating set $S$ of $G$ with $V$ as our open identity neighborhood.
    This provides a finite set $F$ and $N \in \mathbb{N}$
    such that $S \subset (FV)^N$.
    We may assume that $F$ is symmetric and is disjoint from $V$.
    Since $S \subset (FV)^N$,
    we have that $VFV$ generates $G$.
    Therefore, the graph $\Gamma = \Gamma(V,F)$
    is connected and admits a continuous, vertex-transitive action of $G$ with finitely many orbits of edges (the finite set $F$ is in one-to-one correspondence with the set of oriented edge orbits)
    and vertex stabilizers conjugate to the coarsely bounded subgroup $V$.
    The graph $\Gamma$ is thus a Cayley--Abels--Rosendal graph for $G$.
\end{proof}


Now, suppose that $G$ is locally bounded and Polish
but not boundedly generated.
Then it follows that if $V$ is a coarsely bounded, open subgroup of $G$,
the foregoing also proves the following dichotomy
for vertex-transitive actions of $G$ on graphs $\Gamma$
with $V$ as the stabilizer of a vertex:
either $\Gamma$ has infinitely many orbits of edges,
or $\Gamma$ is disconnected.

\section{Examples: Homeomorphism Groups of Countable Stone Spaces}\label{sec:HomeoCtbleStone}

In this final section, we apply the machinery described above to give a full classification of when homeomorphism groups of countable Stone spaces admit Cayley-Abels-Rosendal graphs. This will allow us to obtain a full classification of when these groups are (1) coarsely bounded, (2) locally bounded, and/or (3) boundedly generated.

\begin{DEF}
A \emph{Stone space} is a topological space that is compact, Hausdorff, and totally disconnected.
\end{DEF}

When $X$ is a second countable Stone space,
the group $\Homeo(X)$ equipped with the compact--open topology turns out to be non-Archimedean and Polish.
Two of the present authors~\cite{BranmanLyman} exhibit $\Homeo(X)$ as the automorphism group of a countable graph,
whose vertices are (certain) partitions of $X$ into two clopen (closed and open) sets $U$ and $V$.
The stabilizer of such a vertex is the subgroup of $\Homeo(X)$
either preserving $U$ and $V$ each setwise,
or exchanging them.
Since the intersection of clopen sets is clopen,
a basis for the compact--open topology on $\Homeo(X)$
then is given by $U_\mathcal{P}$,
where $\mathcal{P}$ is a finite partition
$\mathcal{P} = P_0 \sqcup P_1 \sqcup \cdots \sqcup P_n$
of $X$ into clopen sets $P_i$,
and $f \in \Homeo(X)$ belongs to $U_\mathcal{P}$
when it permutes the $P_i$.

\subsection{Countable Stone Spaces}\label{ssec:ctblestonespaces}

For the remainder of this paper we will be concerned with \emph{countable} Stone spaces. See \Cref{fig:stoneexamples} for some examples of countable Stone spaces. These form a particularly nice class of Stone spaces as they are exactly classified by a pair $(\alpha,n)$ where $\alpha$ is a countable ordinal and $n \in \N$. In fact, a consequence of this is that any countable Stone space is exactly homeomorphic to the countable ordinal $\omega^{\alpha}\cdot n +1$ equipped with the order topology. This was first proven by Mazurkiewicz and Sierpi\'{n}ski \cite{MS1920}, but for the sake of completeness we provide a proof here. 

First we need to introduce a type of ``derivative" map on topological spaces. We refer the reader to \cite[Section 6.C]{Kechris1995} and \cite{Milliet2011} for a more thorough treatment of what follows.

\begin{DEF}\label{def:derivedset}
    The \textbf{derived set} of a topological space $X$ is the set of all accumulation points of $X$. We denote the derived set of $X$ by $X'$. For an ordinal $\alpha$ we define the \textbf{$\alpha$-th Cantor-Bendixson derivative}, $X^{\alpha}$, of $X$ recursively as
    \begin{itemize}
        \item $X^{0} = X$, 
        \item $X^{\alpha+1} = (X^{\alpha})'$, and 
        \item $X^{\lambda} =\bigcap_{\alpha < \lambda} X^{\alpha}$ if $\lambda$ is a limit ordinal. 
    \end{itemize}
\end{DEF}

Next we check that for a second countable space, the Cantor-Bendixson derivatives eventually stabilize at some \emph{countable} ordinal. 

\begin{thm}\cite[Theorem 6.9]{Kechris1995}\label{thm:ctbleCanBenrank}
    If a topological space $X$ is second countable, then there exists some countable ordinal $\rho$ such that $X^{\rho} = X^{\rho+1}$. 
\end{thm}

\begin{proof}
    Let $\{U_{n}\}_{n \in \N}$ be a countable open basis for $X$. Given any closed set $F \subset X$, set $N(F) = \{n : U_{n} \cap F \neq \emptyset\}$. Since $X \setminus F = \bigcup_{n \notin N(F)} U_{n}$, the map $F \mapsto N(F)$ is injective into the power set of the natural numbers. Also note that if $F_{1} \subseteq F_{2}$ are two closed subsets, then $N(F_{1}) \subseteq N(F_{2})$. 

    If we assume, to the contrary, that $X^\rho$ does not stabilize for any countable ordinal $\rho$, then by considering the transfinite sequence $(X^{\rho})_{\rho < \omega_{1}}$, where $\omega_{1}$ is the smallest uncountable ordinal, one obtains a monotonic transfinite sequence $(N(X^{\rho}))_{\rho < \omega_{1}}$ of uncountably many subsets of $\N$, a contradiction. 
\end{proof}

Let $X$ be a countable Stone space and $\rho_{X}$ be the smallest countable ordinal such that $X^{\rho_{X}}=X^{\rho_{X}+1}$. Now, $X^{\rho_{X}} = \emptyset$ as $X^{\rho_{X}}$ has no isolated points and is still countable (see, for example, \cite{munkres2000topology}, Theorem 27.7). Additionally, $X$ is compact and hence every family of closed subsets with the finite intersection property has nonempty intersection. Therefore, if $\lambda$ is a limit ordinal with $X^{\lambda} = \bigcap_{\alpha < \lambda} X^{\alpha} = \emptyset$, then for some $\alpha$ we must have $X^{\alpha} = \emptyset$. These two facts together imply that we have $\rho_{X} = \alpha_{X} + 1$ for some countable ordinal $\alpha_{X}$.

We will refer to $\alpha_{X}$ as the \emph{Cantor-Bendixson rank} of $X$. Since every point of $X^{\alpha_X}$ is isolated, there must be finitely many points, and we set $n_{X} = |X^{\alpha_{X}}| \in \N$. We often refer to the $n_{X}$ points of $X^{\alpha_{X}}$ as the set of \emph{maximal points} of $X$. We say that the \emph{characteristic pair} of $X$ is the pair $(\alpha_{X},n_{X})$. We say that a point $x \in X$ has \emph{rank} $\alpha$ if $x \in X^{\alpha} \setminus X^{\alpha+1}$. Note that every point $x \in X$ of rank $\alpha$ has a clopen neighborhood $U_{x} \subset X$ with characteristic pair exactly $(\alpha,1)$. This follows from the fact that $x$ having rank $\alpha$ implies that $x$ is not a limit point of $X^{\alpha}$. 

Finally, we are ready to state the classification theorem of countable Stone spaces. 

\begin{thm}\cite[Th\'{e}or\`{e}me 1]{MS1920} \label{thm:ctblestoneclassification}
    Two countable Stone spaces, $X$ and $Y$, are homeomorphic if and only if they have the same characteristic pair, i.e. $(\alpha_{X},n_{X}) = (\alpha_{Y},n_{Y})$. 
\end{thm}

\begin{proof}
    If $X$ and $Y$ are homeomorphic, then $X^{\rho} \cong Y^{\rho}$ for any countable ordinal $\rho$. In particular, we must have that $\alpha_{X} = \alpha_{Y}$ and $X^{\alpha_{X}} \cong Y^{\alpha_{Y}}$ so that $n_{X} = n_{Y}$ as well. 

    We will prove the reverse implication by transfinite induction on $\alpha$. For the base case, if $\alpha = 0$, then $X$ and $Y$ are simply finite sets and hence homeomorphic if and only if they are of the same cardinality. Now, let $X$ and $Y$ be two countable Stone spaces with characteristic pair $(\alpha_{0},1)$. Our induction hypothesis will be that for any $\alpha < \alpha_{0}$, two spaces with characteristic pair $(\alpha,1)$ are homeomorphic.

    Let $x_{0}\in X$ and $y_{0} \in Y$ be the maximal points of $X$ and $Y$, respectively. Let $\{U_{n}\}_{n \in \N}$ and $\{V_{n}\}_{n \in \N}$ be sequences of clopen subsets such that $\bigcap U_{n} = \{x_{0}\}$ and $\bigcap V_{n} = \{y_{0}\}$. Without loss of generality we assume that $U_{1} = X$ and $V_{1} = Y$. Now set 
    \begin{align*}
        A_{n} &\defeq U_{n} \setminus U_{n+1}, \text{ and } \\
        B_{n} &\defeq V_{n} \setminus V_{n+1}.
    \end{align*}
    Note that each $A_{n}$ and $B_{n}$ is clopen. Let $(\alpha_{n},a_{n})$ and $(\beta_{n},b_{n})$ be the sequences of characteristic pairs of $\{A_{n}\}$ and $\{B_{n}\}$, respectively. Up to replacing a given $A_{n}$ or $B_{n}$ with a finite disjoint union of the $A_{n}$'s or $B_{n}$'s, respectively, we may assume without loss of generality that for each $n$, either $\alpha_{n} < \beta_{n}$ or $a_{n} < b_{n}$, and either $\beta_{n} < \alpha_{n+1}$ or $b_{n}<b_{n+1}$. 

    We now use the induction hypothesis to construct a homeomorphism $f: X \rightarrow Y$. Note that since one of $\beta_{1}$ or $b_{1}$ is larger than $\alpha_{1}$ or $a_{1}$, $Y$ contains a clopen subset, $B'_{1}$, with characteristic pair $(\alpha_{1},a_{1})$. Thus, by the induction hypothesis, there exists a homeomorphism $f_{1}:A_{1} \rightarrow B_{1}'$. Similarly, there exists an $A_{2}' \subset A_{2}$ that is homeomorphic to $B_{2} \setminus B_{2}'$. This allows us to define a homeomorphism $g_{1}:B_{1} \setminus B_{1}' \rightarrow A_{2}'$. We continue this ``back-and-forth" process recursively to obtain, for all $n$, two homeomorphisms
    \begin{align*}
        f_{n}: A_{n}\setminus A_{n}' &\rightarrow B_{n}' \\
        g_{n}:B_{n} \setminus B_{n}' &\rightarrow A_{n+1}'.
    \end{align*}
    Note that the collection of maps of the form $f_{n}$ and $g_{n}^{-1}$ have disjoint support and the union of their supports is exactly $X \setminus \{x\}$. Additionally, the union of the images is exactly $Y \setminus \{y\}$. Thus we define the map
    \begin{align*}
        f(x) \defeq \begin{cases} 
            y_{0} &\text{ if } x=x_{0} \\
            f_{n}(x) &\text{ if } x \in A_{n}\setminus A_{n}' \\
            g_{n}^{-1}(x) &\text{ if } x \in A_{n+1}'
        \end{cases}
    \end{align*}
    where we take $A_{1}' = \emptyset$. Note that since each of the $f_{n}$ and $g_{n}$ are bijections, $f$ is also a bijection. 

    To finish we only need to verify that $f$ is continuous. As $X$ is compact, Hausdorff, and second countable it is metrizable and hence it suffices to check sequential continuity. Let $x_{i} \rightarrow x$ be a convergent sequence in $X$. The argument now splits into two cases.
    
    If $x$ has rank strictly less than $\alpha_{0}$, then $x$ is contained in one of the $A_{n}$. Therefore, throwing out finitely many terms of the sequence, we have that $f$ restricted to the sequence $(x_{i})$ is exactly equal to one of $f_{n}$ or $g_{n-1}^{-1}$ restricted to $(x_{i})$. However, now we must have that $f(x_{i}) \rightarrow x_{i}$ by the continuity of each of $f_{n}$ and $g_{n-1}^{-1}$. 

    Next we suppose that $x = x_{0}$. In this case the fact that $f(x_{i}) \rightarrow f(x) = y$ follows because we chose the collections $\{U_{n}\}$ and $\{V_{n}\}$ such that $\bigcap U_{n} = \{x\}$ and $\bigcap V_{n} = \{y\}$. We thus conclude that $f$ is a continuous bijection and hence a homeomorphism. 
    
    By transfinite induction, we now have that any two countable Stone spaces with characteristic pairs equal to $(\alpha,1)$, for any countable ordinal $\alpha$, are homeomorphic. To conclude the result for arbitrary $n$ we simply note that a space with characteristic pair $(\alpha,n)$ can be partitioned into exactly $n$ clopen subsets each with characteristic pair $(\alpha,1)$. 

    The final statement follows from the fact that the countable ordinal $\omega^{\alpha}\cdot n +1$ equipped with the order topology is exactly a countable Stone space with Cantor-Bendixson rank $\alpha$ and $n$ maximal points. 
\end{proof}

Following this theorem, from now on we write $X_{\alpha,n}$ to refer to the unique (up to homeomorphism) countable Stone space with Cantor-Bendixson rank $\alpha$ and $n$ maximal points. That is, the countable Stone space $X_{\alpha,n}$ is homeomorphic, but not necessarily order isomorphic, to the countable ordinal $\omega^{\alpha}\cdot n +1$ equipped with the order topology. See \Cref{fig:stoneexamples} for some examples of these spaces. We no longer need to take derived sets, so from here on the prime symbol, $'$, is used to decorate notation.

\begin{figure}[ht!]\centering
\begingroup%
  \makeatletter%
  \providecommand\color[2][]{%
    \errmessage{(Inkscape) Color is used for the text in Inkscape, but the package 'color.sty' is not loaded}%
    \renewcommand\color[2][]{}%
  }%
  \providecommand\transparent[1]{%
    \errmessage{(Inkscape) Transparency is used (non-zero) for the text in Inkscape, but the package 'transparent.sty' is not loaded}%
    \renewcommand\transparent[1]{}%
  }%
  \providecommand\rotatebox[2]{#2}%
  \newcommand*\fsize{\dimexpr\f@size pt\relax}%
  \newcommand*\lineheight[1]{\fontsize{\fsize}{#1\fsize}\selectfont}%
  \ifx\svgwidth\undefined%
    \setlength{\unitlength}{346.67356236bp}%
    \ifx\svgscale\undefined%
      \relax%
    \else%
      \setlength{\unitlength}{\unitlength * \real{\svgscale}}%
    \fi%
  \else%
    \setlength{\unitlength}{\svgwidth}%
  \fi%
  \global\let\svgwidth\undefined%
  \global\let\svgscale\undefined%
  \makeatother%
  \begin{picture}(1,0.94595513)%
    \lineheight{1}%
    \setlength\tabcolsep{0pt}%
    \put(0,0){\includegraphics[width=\unitlength,page=1]{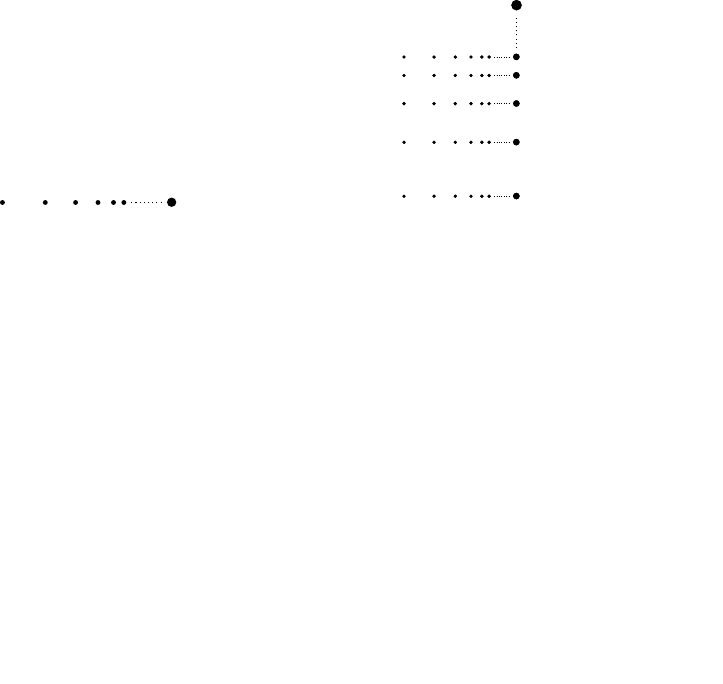}}%
    \put(0.05756282,0.72687188){\color[rgb]{0,0,0}\makebox(0,0)[lt]{\lineheight{1.25}\smash{\begin{tabular}[t]{l}$X_{1,1}$\end{tabular}}}}%
    \put(0.7807818,0.85561789){\color[rgb]{0,0,0}\makebox(0,0)[lt]{\lineheight{1.25}\smash{\begin{tabular}[t]{l}$X_{2,3}$\end{tabular}}}}%
    \put(0,0){\includegraphics[width=\unitlength,page=2]{StoneExamples.pdf}}%
    \put(0.06478316,0.36019628){\color[rgb]{0,0,0}\makebox(0,0)[lt]{\lineheight{1.25}\smash{\begin{tabular}[t]{l}$X_{\omega,2}$\end{tabular}}}}%
    \put(0,0){\includegraphics[width=\unitlength,page=3]{StoneExamples.pdf}}%
  \end{picture}%
\endgroup%

		\caption{Three examples of countable Stone spaces as subspaces of the plane. The top two figures are examples of spaces whose rank is a successor ordinal and the bottom is an example of a limit ordinal. In the bottom figure, the number $n \in \Z_{\geq 0}$ is used to represent a small copy of $X_{n,1}$.}
        \label{fig:stoneexamples}
\end{figure}

Our goal is to now prove the classification theorem stated in the introduction, recalled below for convenience. 

\begin{customthm}{A} 
    Let $X_{\alpha,n}$ be a countable Stone space. The group $\Homeo(X_{\alpha,n})$ is always locally bounded and is
    \begin{itemize}
        \item coarsely bounded if and only if $n=1$, and 
        \item boundedly generated but not coarsely bounded if and only $n > 1$ and $\alpha$ is a successor ordinal.
    \end{itemize}
\end{customthm}

We break this into several steps. The first is to verify that when $n=1$, the homeomorphism group is coarsely bounded. Next we will see that this immediately implies that all of these groups are locally bounded. We will then build Cayley-Abels-Rosendal graphs when $\alpha$ is a sucessor ordinal, thus proving bounded generation of $\Homeo(X_{\alpha,n})$. Finally, we will consider the limit ordinal case.

\subsection{Coarsely Bounded and Locally Bounded}

We begin with the case that $n=1$. 
Notice that since $X_{\alpha,1}$ has a unique maximal point,
any finite clopen partition of $X_{\alpha,1}$,
say as $P_0 \sqcup P_1 \sqcup \cdots \sqcup P_n$,
has the property that (after possibly relabeling)
$P_0$ contains this maximal point
and is thus by the classification of countable Stone spaces,
actually homeomorphic to $X_{\alpha,1}$.

We will address ourselves to more general Stone spaces
that retain this important property:
following Mann--Rafi~\cite{MR2023},
say that a Stone space $X$
is \emph{self-similar} if every finite clopen partition
$X = P_0 \sqcup P_1 \sqcup \cdots \sqcup P_n$
has the property that (after possibly relabeling)
$P_0$ has a clopen subset homeomorphic to the whole space $X$.

Recall that by our description of the compact--open topology on $\homeo(X)$, every identity neighborhood $U$ in $\homeo(X)$ contains an open subgroup $V = \Stab(\mathcal{P})$
comprising those elements of $\homeo(X)$
stabilizing some finite clopen partition $\mathcal{P}$ of $X$.

\begin{prop}\label{CBbygraphs}
Suppose $X$ is a self-similar Stone space,
and that $U$ is an open subset of $\Homeo(X)$.
Let $\mathcal{P} = P_0 \sqcup P_1 \sqcup \cdots \sqcup P_n$
be a clopen partition of $X$
such that $U$ contains $V = \Stab(\mathcal{P})$
and such that $P_0$ contains a clopen subset
homeomorphic to $X$.
There exists an involution $f \notin V$
such that the graph $\Gamma(V,F)$ as defined in \Cref{sec:CARgraphs}
associated to $V$ and to $F = \{f\}$
has diameter at most three (so is connected)
and admits a vertex-transitive and edge-transitive action of $\homeo(X)$.
\end{prop}

\begin{proof}
We first construct a graph $\Gamma$ associated to the given partition $\mathcal{P}$ that admits a vertex-transitive and edge-transitive action of $\Homeo(X)$.
We show that the graph $\Gamma$ has diameter at most three, and then realize it as combinatorially identical to the graph $\Gamma(V,F)$ as defined in \Cref{sec:CARgraphs} associated to the subgroup $V = \Stab(\mathcal{P})$ and the set $F = \{f\}$ for some involution $f \notin V$.

For the vertex set of $\Gamma$, take the set of partitions $\mathcal{Q} = Q_0 \sqcup Q_1 \sqcup \cdots \sqcup Q_n$,
where each $Q_i$ is homeomorphic to $P_i$.
If $\mathcal{Q}$ and $\mathcal{Q}'$ are partitions,
declare the associated vertices to be adjacent
when $Q_1 \sqcup \cdots \sqcup Q_n \subset Q'_0$,
from which it follows that $Q'_1 \sqcup \cdots \sqcup Q'_n \subset Q_0$.
Notice that this is actually made possible only by the assumption that $Q_0$ contains a clopen subset homeomorphic to $X$.

Now if $\mathcal{Q}$ and $\mathcal{Q}'$ are adjacent,
then provided $i > 0$,
we have that $Q_i$ and $Q'_i$ are disjoint,
and there is an obvious involution 
which exchanges each $Q_i$ with $Q'_i$
and fixes the complement $Q_0 \cap Q'_0$ pointwise.

Let us show that the graph so described has diameter at most three.
Given two partitions $\mathcal{Q}$ and $\mathcal{Q'}$,
if $Q_1\sqcup\cdots\sqcup Q_n$ is disjoint
from $Q'_1\sqcup \cdots\sqcup Q'_n$, then the partitions
are adjacent as vertices of $\Gamma$.
If not, observe that because $Q_0$ and $Q'_0$
each contain clopen subsets homeomorphic to $X$,
either $Q_0 \cap Q'_0$ does too,
or if not, then \emph{both} $Q_0 - Q'_0$ and $Q'_0 - Q_0$ do.
We will show that in the former case,
$\mathcal{Q}$ and $\mathcal{Q}'$ have a common neighbor,
while in the latter there exists
$\mathcal{R}$ adjacent to $\mathcal{Q}$
and $\mathcal{R}'$ adjacent to $\mathcal{Q}'$
which are adjacent.

Now in the former case,
$Q_0 \cap Q'_0$ contains
disjoint clopen copies of each $P_i$ for $i > 0$;
call each copy $Q''_i$,
and set $Q''_0$ to be the complement
of $Q''_1 \sqcup \cdots \sqcup Q''_n$.
By construction,
we see that this partition $\mathcal{Q}''$
is adjacent to both $\mathcal{Q}$ and $\mathcal{Q}''$.

In the latter case,
we proceed similarly
by finding copies of $P_i$,
say $R_i \subset Q_0 - Q'_0$
and $R'_i \subset Q'_0 - Q_0$
and setting $R_0$ to be the complement
of $R_1 \sqcup \cdots \sqcup R_n$
and $R'_0$ the complement of $R'_1 \sqcup \cdots \sqcup R'_n$.
The partition $\mathcal{R}$
is adjacent to $\mathcal{Q}$ by construction,
and $\mathcal{R}'$ is adjacent to $\mathcal{Q}'$,
and because $Q_0 - Q'_0$ is disjoint from $Q'_0 - Q_0$,
the partitions $\mathcal{R}$ and $\mathcal{R}'$
are themselves adjacent.
Together these cases demonstrate that the graph $\Gamma$
has diameter at most three.

The group $\Homeo(X)$
acts transitively on the vertices of $\Gamma$
by construction.
It also acts transitively on the edges:
given vertex transitivity,
it suffices to show that if $\mathcal{Q}$ and $\mathcal{Q'}$ are partitions
corresponding to vertices which are each adjacent to a common vertex represented by a partition $\mathcal{P}$,
then there is a homeomorphism of $X_{\alpha,1}$
taking $\mathcal{Q}$ to $\mathcal{Q}'$
while fixing $\mathcal{P}$.
Such a homeomorphism exists:
writing $\mathcal{P}$ as $P_0 \sqcup P_1 \sqcup \cdots \sqcup P_n$,
we have by assumption that $Q_i$ and $Q'_i$
are contained in $P_0$ for $i > 0$,
so a homeomorphism $f$ taking each $Q_i$ to $Q'_i$
for $i > 0$ may be chosen which fixes each $P_j$
pointwise for $j > 0$.
The homeomorphism $f$ thus also preserves
$P_0$ setwise,
hence fixes the vertex of $\Gamma$ corresponding to $\mathcal{P}$.

Finally, we see that $\Gamma$
may be identified with the graph $\Gamma(V,F)$
as defined in \Cref{sec:CARgraphs}
associated to the subgroup $V = \Stab(\mathcal{P})$
and the set $F = \{f\}$ for some involution $f \notin V$.
Indeed, a homeomorphism of $X$ fixes the vertex associated to $\mathcal{P}$ precisely when it is contained in $V$,
and we saw that any two adjacent vertices are exchanged by some involution and that $\Gamma$
has one $\Homeo(X)$-orbit of edges,
so choosing for $f$ any such involution not contained in $V$ (for example, one which exchanges $\mathcal{P}$ itself with an adjacent partition) proves the claim.
\end{proof}

\begin{figure}[ht!]\centering
        \def\svgwidth{\textwidth}
\begingroup%
  \makeatletter%
  \providecommand\color[2][]{%
    \errmessage{(Inkscape) Color is used for the text in Inkscape, but the package 'color.sty' is not loaded}%
    \renewcommand\color[2][]{}%
  }%
  \providecommand\transparent[1]{%
    \errmessage{(Inkscape) Transparency is used (non-zero) for the text in Inkscape, but the package 'transparent.sty' is not loaded}%
    \renewcommand\transparent[1]{}%
  }%
  \providecommand\rotatebox[2]{#2}%
  \newcommand*\fsize{\dimexpr\f@size pt\relax}%
  \newcommand*\lineheight[1]{\fontsize{\fsize}{#1\fsize}\selectfont}%
  \ifx\svgwidth\undefined%
    \setlength{\unitlength}{496.66724017bp}%
    \ifx\svgscale\undefined%
      \relax%
    \else%
      \setlength{\unitlength}{\unitlength * \real{\svgscale}}%
    \fi%
  \else%
    \setlength{\unitlength}{\svgwidth}%
  \fi%
  \global\let\svgwidth\undefined%
  \global\let\svgscale\undefined%
  \makeatother%
  \begin{picture}(1,0.45055636)%
    \lineheight{1}%
    \setlength\tabcolsep{0pt}%
    \put(0,0){\includegraphics[width=\unitlength,page=1]{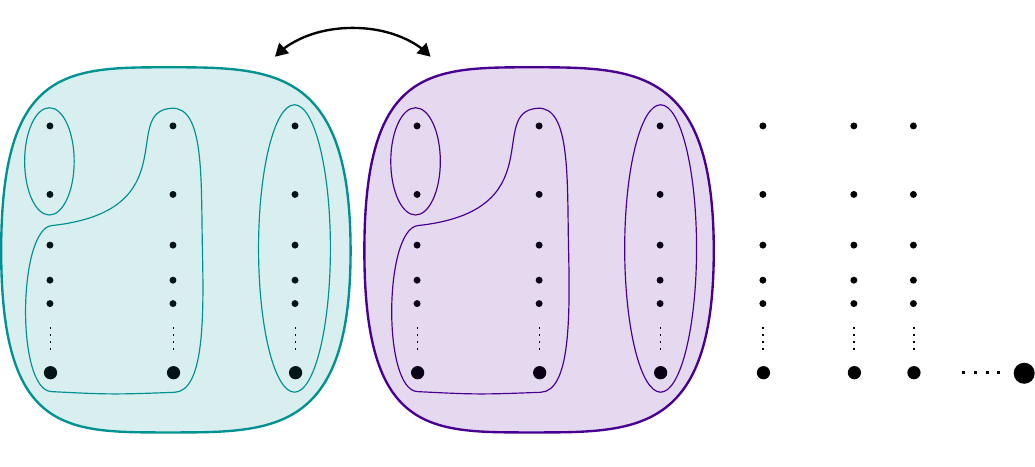}}%
    \put(0.33397381,0.43598356){\color[rgb]{0,0,0}\makebox(0,0)[lt]{\lineheight{1.25}\smash{\begin{tabular}[t]{l}$f$\end{tabular}}}}%
    \put(0.1617842,0.00281904){\color[rgb]{0,0,0}\makebox(0,0)[lt]{\lineheight{1.25}\smash{\begin{tabular}[t]{l}$P'$\end{tabular}}}}%
    \put(0.51692536,0.00281904){\color[rgb]{0,0,0}\makebox(0,0)[lt]{\lineheight{1.25}\smash{\begin{tabular}[t]{l}$P''$\end{tabular}}}}%
  \end{picture}%
\endgroup%

		\caption{An example of the sets $P'$ and $P''$ and the involution $f$ for $X_{2,1}$. The subsets of $P'$ are meant to represent the partition sets $P_{1},P_{2},$ and $P_{3}$. Note that $P_{0}$ is simply the complement of $P'$ and hence not drawn explicitly.}
        \label{fig:CBProof1}
\end{figure}
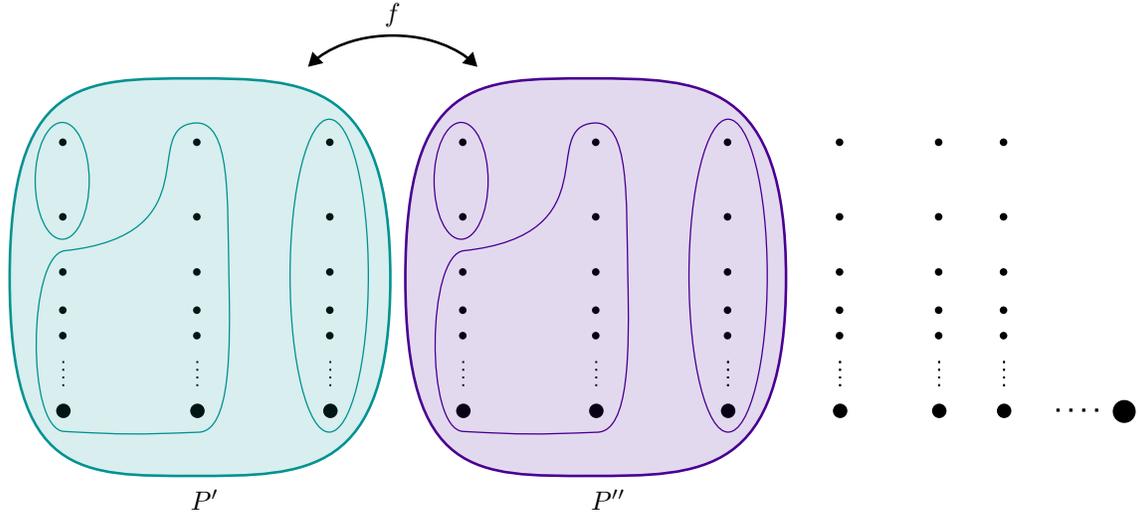

This proposition gives us the desired corollary.

\begin{cor}\label{prop:degree1CB}
Supposing that the Stone space $X$ is self-similar,
the group $\homeo(X)$ is coarsely bounded.
In particular if $\alpha$ is a countable ordinal,
the group $\Homeo(X_{\alpha,1})$ is coarsely bounded.
\end{cor}

\begin{proof}
We first prove that
every orbit in every continuous action of $\homeo(X)$
by isometries on a metric space
is Lipschitz-dominated by an action on a
connected graph as in the statement of \Cref{CBbygraphs}. Suppose that $\homeo(X)$
acts continuously and by isometries on a metric space $X$.
Take for $U$ in the statement above the open set comprising
those elements failing to move a point $x \in X$
a distance more than $\epsilon > 0$,
and construct the graph $\Gamma(V,F)$
with vertex stabilizer associated to $V$ and $F$.
Choosing a transversal of coset representatives $\mathcal{T}$ of $V$ defines a map $\Phi\colon V\Gamma \to X$ as $gV \mapsto g.x$ for $g \in \mathcal{T}$.
Write $\delta$ for $d(x, f.x)$,
where $f$ is the involution in the statement of the proposition.
The map $\Phi$
is Lipschitz: Let $gV$ and $hV$ be connected by an edge in $\Gamma(V,F)$ so that $Vg^{-1}hV = VfV = VFV$. Then 
\begin{align*}
    d(\Phi(gV),\Phi(hV)) &= d(g.x,h.x) \\
    &= d(x,g^{-1}h.x) \\
    &\leq 2\epsilon + \delta.
\end{align*} 


Since every such graph $\Gamma(V,F)$ has diameter at most three,
we conclude.
\end{proof}

We note that the above corollary also follows from \cite[Proposition 3.1]{MR2023} by considering $\homeo(X)$ as a continuous quotient of the mapping class group of the genus zero surface with end space homeomorphic to $X$.

\Cref{prop:degree1CB} has as an immediate corollary that all of the homeomorphism groups we are considering are locally bounded. 

\begin{cor} \label{cor:locallycb}
    The group $\Homeo(X_{\alpha,n})$ is locally bounded.
\end{cor}

\begin{proof}
    Consider a partition of $X_{\alpha,n}$ given by $\cP = P_{1}\sqcup \cdots \sqcup P_{n}$ where each $P_{i}$ contains exactly one maximal point. 
    
    Now, the stabilizer $\Stab(\cP)$ is open
    and sits in a (continuous) short exact sequence

    \[
    \begin{tikzcd}
        1 \ar[r] & \prod_{i=1}^n \Homeo(X_{\alpha,1})
        \ar[r] & \Stab(\cP) \ar[r] & S_n \ar[r] & 1.
    \end{tikzcd}
    \]
    (Here the topology on the finite symmetric group
    $S_n$ is discrete.)
    One can show directly that the topological group
    $\Stab(\cP)$ is coarsely bounded,
    since any continuous action by isometries on a metric space with unbounded orbits
    would produce an action of either the kernel
    (and hence the coarsely bounded group $\Homeo(X_{\alpha,1})$)
    or the quotient with unbounded orbits,
    but both are coarsely bounded.
    Thus $\Stab(\cP)$ is a coarsely bounded identity neighborhood in $\Homeo(X_{\alpha,n})$.
\end{proof}

\subsection{Cayley-Abels-Rosendal Graphs for Successor Ordinals}\label{ssec:successorordinals}

Next we turn to the case when $n\geq 2$ and $\alpha$ is a successor ordinal. We will build an unbounded Cayley-Abels-Rosendal graph for $\Hoan$, proving the following.

\begin{prop}\label{prop:successor}
    Let $\alpha$ be a successor ordinal and $n$ an integer satisfying $n \ge 2$. Then $\Homeo(X_{\alpha,n})$ is boundedly generated and not coarsely bounded. 
\end{prop}

Suppose that $X_{\alpha,n}$
is a countable Stone space,
where \(n \ge 2\) and \(\alpha = \beta + 1\) is a successor ordinal.
The space \(X_{\alpha,n}\) therefore has \(n\) maximal points.
Say that a partition \(\mathcal{P}\) is \emph{good} when it comprises exactly \(n\) clopen sets,
each containing a single maximal point.
Each clopen set in the partition is therefore homeomorphic to \(X_{\alpha,1}\).

We consider the operation of \emph{shifting} a good partition \(\mathcal{P}\):
choose a pair of maximal points \(x_i\) and \(x_j\),
write \(P_i\) and \(P_j\) for the clopen sets in the partition containing \(x_i\) and \(x_j\) respectively.
Remove from \(P_i\) a clopen subset homeomorphic to \(X_{\beta,1}\)
and add it to \(P_j\).
We say that two such partitions \emph{differ by a maximal shift.}

The prototypical example is when \(X \cong X_{2,1}\)
is the end compactification of \(\mathbb{Z}\).
Here \(\beta = 0\), so sets homeomorphic to \(X_{\beta,1}\) are single points.
One pair of good partitions \(\mathcal{P}\) and \(\mathcal{Q}\)
is given by \(\mathcal{P}=[-\infty,0] \sqcup [1,\infty]\)
and \(\mathcal{Q}=[-\infty,-1,] \sqcup [0,\infty]\).
The operation “add one” on \(\mathbb{Z}\) extends to a homeomorphism of the end compactification
which takes \(\mathcal{Q}\) to \(\mathcal{P}\).

Consider the graph \(\Gamma = \Gamma(\alpha,n)\)
whose vertices are the good partitions of \(X\),
where two vertices are connected by an edge when the corresponding partitions
differ by a maximal shift. See \Cref{fig:successgraphex} for an example. 

    \begin{figure}[ht!]\centering
        \def\svgwidth{\textwidth}
        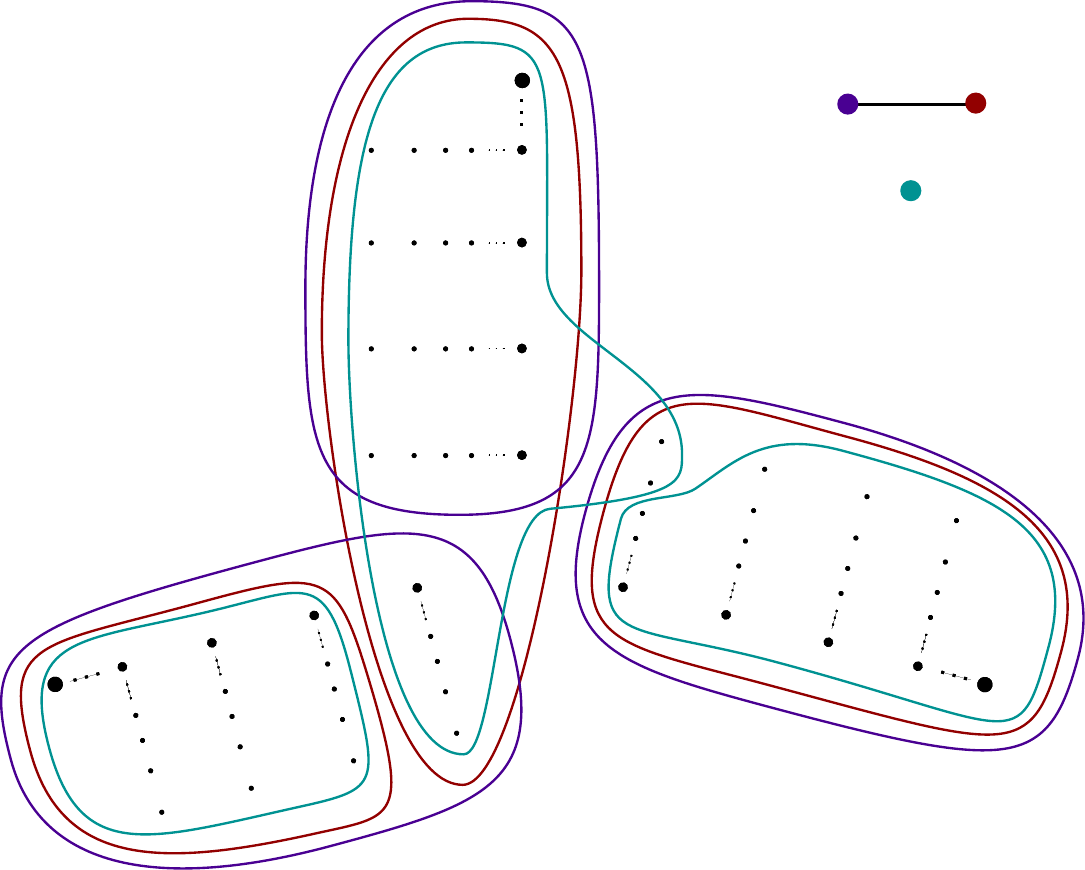
		\caption{An example of three partitions in $\G = \G(2,3)$. The red and purple vertices are connected by an edge as they differ by a maximal shift. However, the teal vertex is not connected by an edge to either or the red or purple vertices since while it may appear to differ by a ``shift," it does not differ by a maximal shift.} 
        \label{fig:successgraphex}
    \end{figure}

\begin{lem}\label{CARprop:successor}
When \(n \ge 2\) and \(\alpha = \beta + 1\) is a successor ordinal,
the graph \(\Gamma(\alpha,n)\) is connected and has infinite diameter.
\end{lem}

\begin{proof}
Consider two good partitions \(\mathcal{P}\) and \(\mathcal{Q}\).
For each maximal point \(x_i\),
consider the clopen set \(R'_i = P_i \cap Q_i\),
where \(P_i\) and \(Q_i\) are the partition sets in \(\mathcal{P}\) and \(\mathcal{Q}\)
respectively containing \(x_i\).
The complement \(R'_0 = X - \coprod_{i = 1}^n R_i\)
is a clopen set containing none of the maximal points,
so it can contain at most finitely many, say \(d\), points of rank \(\beta\).

We will connect \(\mathcal{P}\) to \(\mathcal{Q}\) by a path of length at most \(d + 2n(n-1)\)
by beginning with \(\mathcal{P}\) and progressively altering it until we have produced \(\mathcal{Q}\).
This process involves considering each ordered pair of maximal points \(x_i\) and \(x_j\).
By only altering \(P_i\) and \(P_j\) by shifts,
we will make it so that the set of points in \(P_i\) which are in \(Q_j\) is empty.
If \(P_i\) contains some points of \(Q_j\),
the number of shifts we will use
is equal to either \(d_{ij}\), the number of rank \(\beta\) points in $P_i\cap Q_j$,
if \(d_{ij} > 0\), or to $2$.

Assuming \(d_{ij} > 0\),
the set of points in \(P_i \cap Q_j\) is homeomorphic to \(X_{\beta,d_{ij}}\);
choosing a ``good'' partition of this set,
i.e. one comprised of $d_{ij}$ sets, each containing exctly one point of rank $\beta$,
we can shift one element of this partition at a time
out of \(P_i\) and into \(P_{j}\),
producing a path of length \(d_{ij}\) from \(\mathcal{P}\)
to a new partition \(\mathcal{P}'\) which satisfies that \(P'_i \cap Q_j = \varnothing\). See \Cref{fig:Gammaconnected1} for an example of this.

\begin{figure}[ht!]\centering
        \def\svgwidth{\textwidth}
\begingroup%
  \makeatletter%
  \providecommand\color[2][]{%
    \errmessage{(Inkscape) Color is used for the text in Inkscape, but the package 'color.sty' is not loaded}%
    \renewcommand\color[2][]{}%
  }%
  \providecommand\transparent[1]{%
    \errmessage{(Inkscape) Transparency is used (non-zero) for the text in Inkscape, but the package 'transparent.sty' is not loaded}%
    \renewcommand\transparent[1]{}%
  }%
  \providecommand\rotatebox[2]{#2}%
  \newcommand*\fsize{\dimexpr\f@size pt\relax}%
  \newcommand*\lineheight[1]{\fontsize{\fsize}{#1\fsize}\selectfont}%
  \ifx\svgwidth\undefined%
    \setlength{\unitlength}{466.73658957bp}%
    \ifx\svgscale\undefined%
      \relax%
    \else%
      \setlength{\unitlength}{\unitlength * \real{\svgscale}}%
    \fi%
  \else%
    \setlength{\unitlength}{\svgwidth}%
  \fi%
  \global\let\svgwidth\undefined%
  \global\let\svgscale\undefined%
  \makeatother%
  \begin{picture}(1,0.40913116)%
    \lineheight{1}%
    \setlength\tabcolsep{0pt}%
    \put(0,0){\includegraphics[width=\unitlength,page=1]{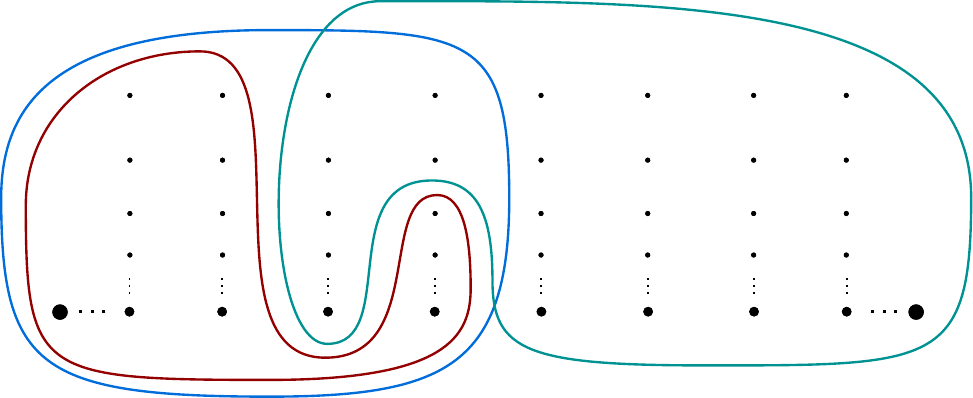}}%
    \put(0.85866411,0.39240103){\color[rgb]{0,0,0}\makebox(0,0)[lt]{\lineheight{1.25}\smash{\begin{tabular}[t]{l}$Q_{2}$\end{tabular}}}}%
    \put(0.14168949,0.39240103){\color[rgb]{0,0,0}\makebox(0,0)[lt]{\lineheight{1.25}\smash{\begin{tabular}[t]{l}$P_{1}$\end{tabular}}}}%
    \put(0.06328815,0.26854963){\color[rgb]{0,0,0}\makebox(0,0)[lt]{\lineheight{1.25}\smash{\begin{tabular}[t]{l}$P_{1}'$\end{tabular}}}}%
  \end{picture}%
\endgroup%

		\caption{An example of two partitions $\cP$, in blue, and $\cQ$, in teal, with $d_{12} =1$ in $\G(2,2)$. Pictured in red is $P_{1}'$ obtained by shifting $P_{1} \cap Q_{2}$ out of $P_{1}$ and into $P_{2}$. Note that the left maximal point is $x_{1}$ and the right is $x_{2}$.} 
        \label{fig:Gammaconnected1}
\end{figure}
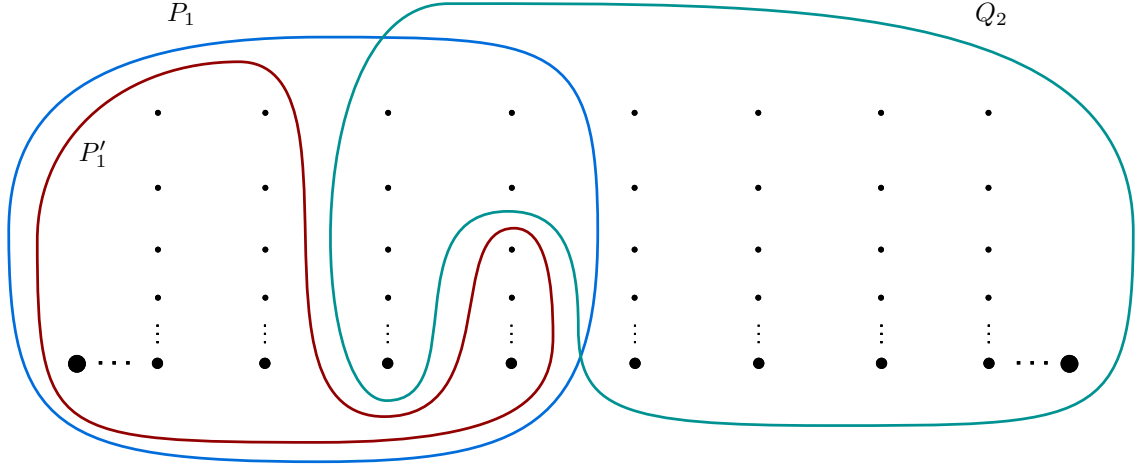

If instead \(d_{ij} = 0\),
we may freely choose a single rank \(\beta\) point \(y\) of \(P_i\).
This point \(y\) has two clopen neighborhoods \(U\) and \(V\), both homeomorphic to \(X_{\beta,1}\),
such that \(U = V \sqcup (P_i \cap Q_j)\).
First shifting \(U\) out of \(P_i\) and into \(P_{j}\), and then shifting \(V\) back in
produces a path of length two in \(\Gamma\)
between \(\mathcal{P}\) and a new partition \(\mathcal{P}'\)
which satisfies that \(P'_i \cap Q_j = \varnothing\). See \Cref{fig:Gammaconnected2} for an example.

\begin{figure}[ht!]\centering
        \def\svgwidth{\textwidth}
        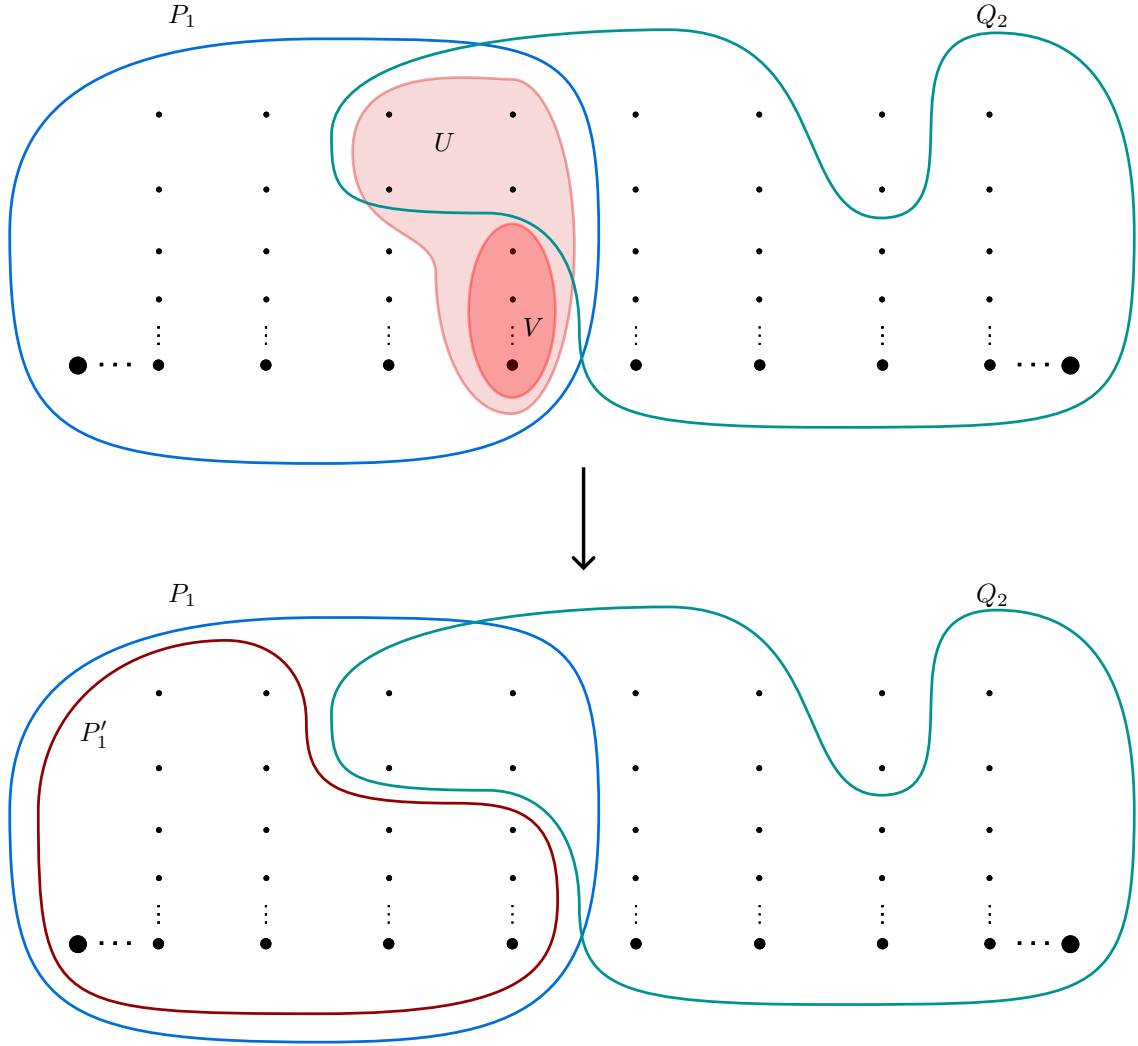
		\caption{An example of two partitions $\cP$, in blue, and $\cQ$, in teal, with $d_{12} =0$ in $\G(2,2)$. The sets $U$ and $V$ are represented in the top figure. In the bottom figure we have, in red, $P_{1}'$ obtained by shifting $U$ out of $P_{1}$ and into $P_{2}$ and then shifting $V$ back. Again, the left maximal point is $x_{1}$ and the right is $x_{2}$.} 
        \label{fig:Gammaconnected2}
\end{figure}

Repeating this process for each ordered pair \((i,j)\) yields the desired path,
proving that \(\Gamma\) is connected.

To see that the graph \(\Gamma\) has infinite diameter,
notice that because a shift moves one point of rank \(\beta\) at a time,
if the set \(R'_0\) constructed above contains \(d\) points of rank \(\beta\),
any path from \(\mathcal{P}\) to \(\mathcal{Q}\)
must have length at least \(d\).
Any natural number \(d\) is realized as the number of rank $\beta$ points in $R'_0$, so \(\Gamma\) has infinite diameter.
(Put another way, this observation and the existence of the path above
proves that counting the number of rank \(\beta\) points
different between \(\mathcal{P}\) and \(\mathcal{Q}\)
is a coarse measure of their distance in \(\Gamma\).)
\end{proof}

Now, recall that for any good partition $\cP$,
the stabilizer $\Stab(\cP)$ is open and coarsely bounded.
The group $\Homeo(X_{\alpha,n})$
acts continuously and transitively on the set of good partitions.
Fix a good partition $\cP$
and a maximal shift $g \in \Homeo(X_{\alpha,n})$
which moves $\cP$,
and let $F = \{g^{\pm 1}\}$.

\begin{lem}\label{CARprop:successor-CAR}
The graph \(\Gamma\) is of the form \(\Gamma(\Stab(\cP),F)\), where \(F = \{g^{\pm 1}\}\)
and is thus a Cayley–Abels–Rosendal graph for \(\homeo(X_{\alpha,n})\).
\end{lem}

\begin{proof}
By the paragraph above,
the group $\Homeo(X_{\alpha,n})$
acts transitively on $\Gamma(\alpha,n)$
with stabilizers conjugate to $\Stab(\cP)$,
which is open in \(\homeo(X_{\alpha,n})\) and coarsely bounded.
Now, supposing that \(\mathcal{Q}\) and \(\mathcal{Q}'\) are good partitions of \(X\)
which differ from \(\mathcal{P}\) by a shift,
notice that there is a homeomorphism preserving \(\mathcal{P}\)
taking one of the shifted sets, which is homeomorphic to \(X_{\beta,1}\), to the other. See \Cref{fig:Gamma2EdgesOrbit} for an example of such a homeomorphism.
Thus the group \(\homeo(X_{\alpha,n})\) acts edge-transitively on the graph \(\Gamma\),
and $F$ is therefore a complete set of representatives
for the orbits of oriented edges of \(\Gamma\).

\begin{figure}[ht!]\centering
        \def\svgwidth{\textwidth}
\begingroup%
  \makeatletter%
  \providecommand\color[2][]{%
    \errmessage{(Inkscape) Color is used for the text in Inkscape, but the package 'color.sty' is not loaded}%
    \renewcommand\color[2][]{}%
  }%
  \providecommand\transparent[1]{%
    \errmessage{(Inkscape) Transparency is used (non-zero) for the text in Inkscape, but the package 'transparent.sty' is not loaded}%
    \renewcommand\transparent[1]{}%
  }%
  \providecommand\rotatebox[2]{#2}%
  \newcommand*\fsize{\dimexpr\f@size pt\relax}%
  \newcommand*\lineheight[1]{\fontsize{\fsize}{#1\fsize}\selectfont}%
  \ifx\svgwidth\undefined%
    \setlength{\unitlength}{440.43471449bp}%
    \ifx\svgscale\undefined%
      \relax%
    \else%
      \setlength{\unitlength}{\unitlength * \real{\svgscale}}%
    \fi%
  \else%
    \setlength{\unitlength}{\svgwidth}%
  \fi%
  \global\let\svgwidth\undefined%
  \global\let\svgscale\undefined%
  \makeatother%
  \begin{picture}(1,0.4660881)%
    \lineheight{1}%
    \setlength\tabcolsep{0pt}%
    \put(0,0){\includegraphics[width=\unitlength,page=1]{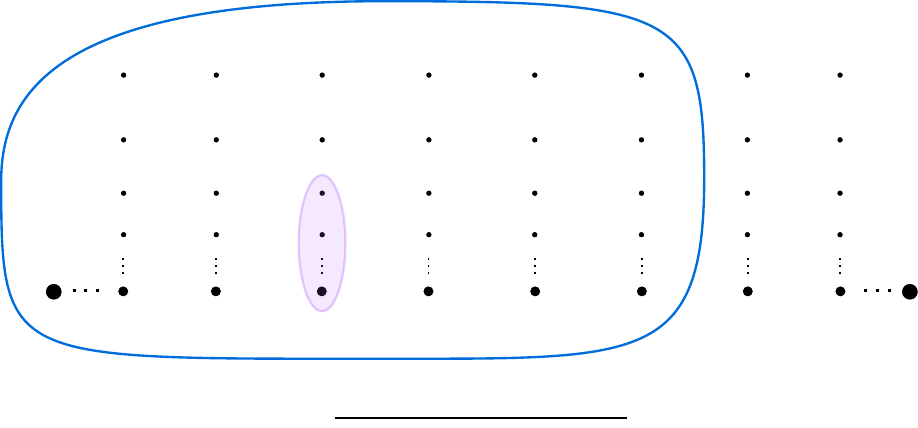}}%
    \put(0.62981918,0.10840986){\color[rgb]{0,0,0}\makebox(0,0)[lt]{\lineheight{1.25}\smash{\begin{tabular}[t]{l}$\tau$\end{tabular}}}}%
    \put(0,0){\includegraphics[width=\unitlength,page=2]{Gamma2EdgesOrbit.pdf}}%
    \put(0.52750505,0.03199192){\color[rgb]{0,0,0}\makebox(0,0)[lt]{\lineheight{1.25}\smash{\begin{tabular}[t]{l}$\mathcal{P}$\end{tabular}}}}%
    \put(0.68644712,0.03199192){\color[rgb]{0,0,0}\makebox(0,0)[lt]{\lineheight{1.25}\smash{\begin{tabular}[t]{l}$\mathcal{R}$\end{tabular}}}}%
    \put(0.36436644,0.03199192){\color[rgb]{0,0,0}\makebox(0,0)[lt]{\lineheight{1.25}\smash{\begin{tabular}[t]{l}$\mathcal{S}$\end{tabular}}}}%
    \put(0,0){\includegraphics[width=\unitlength,page=3]{Gamma2EdgesOrbit.pdf}}%
  \end{picture}%
\endgroup%

		\caption{An example of two edges in $\G(2,2)$ sharing the vertex $\cP$ and the map $\tau \in \Stab(\cP)$ that maps one edge to the other.} 
        \label{fig:Gamma2EdgesOrbit}
\end{figure}
\end{proof}

Interpreting the statements of \Cref{CARprop:successor} and \Cref{CARprop:successor-CAR} proves the classification for successor ordinals given in \Cref{prop:successor}.

\begin{proof}[Proof of \Cref{prop:successor}]
    \sloppy Because the graph $\Gamma(\alpha,n)$ is a Cayley-Abels-Rosendal graph for $\Homeo(X_{\alpha,n})$, the group is boundedly generated. Because the graph has infinite diameter, the group is not coarsely bounded.
\end{proof}

\subsection{Limit Ordinals}\label{ssec:limitordinals}

Finally we turn to limit ordinals and prove the following.

\begin{prop}\label{prop:limitord}
    Let $\alpha$ be a limit ordinal. If $n \geq 2$, then $\Hoan$ is not boundedly generated. 
\end{prop}

We will prove this by using \Cref{lem:countgen}. That is, we will show that $\Hoan$ is a union of a countably infinite chain of proper open subgroups. We again label the maximal points of $X_{\alpha,n}$ as $x_{1},\ldots,x_{n}$ and let $\fP$ be the set of all good partitions. We will define this chain of subgroups by first defining a height function on $\fP$. Then the subgroups will be defined as stabilizers of sublevel sets of this height function. 

\begin{remark}
    Note that the set $\fP$ is defined in the same way as our vertex set in the previous section. One interpretation of the arguments in this section is that we are showing that there is no way to make a Cayley-Abels-Rosendal graph out of this vertex set. Of course, the actual result we prove is stronger. It implies that there cannot exist a Cayley-Abels-Rosendal graph for \emph{any} choice of vertex set. 
\end{remark}

Given any clopen subset $A \subset X_{\alpha,n}$ and ordinal $\beta < \alpha$ we let $[A]_{\beta}$ denote the set of points of type $\beta$ in $A$. We now define a relative height function on partitions as follows. For $\cP,\cQ \in \fP$, 
\begin{align*}
    h(\cP,\cQ) \defeq \sup\left\{\beta \big\vert \beta < \alpha \text{ and } [P_{i}\triangle Q_{i}]_{\beta} \neq \emptyset \text{ for all } i=1,\ldots,n\right\}.
\end{align*}
We first check some basic properties of this function. 

\begin{lem}\label{lem:heightproperties}
    The function $h$ takes values strictly less than $\alpha$, is $\Hoan$-equivariant, and satisfies a strong triangle inequality. That is, for $\cP,\cQ,\cR \in \fP$, 
    \begin{align*}
        h(\cP,\cQ) \leq \max\{h(\cP,\cR),h(\cR,\cQ)\}.
    \end{align*}
\end{lem}

\begin{proof}
    We first check that $h(\cP,\cQ) < \alpha$ for all $\cP,\cQ \in \fP$. Suppose, to the contrary, that it is not. Then, after passing to a subsequence, there would exist an increasing sequence $\beta_{1}<\beta_{2}<\cdots$ such that $\beta_{k} \rightarrow \alpha$ and an $i$ so that $[P_{i} \setminus Q_{i}]_{\beta_{k}}\neq \emptyset$ for all $k$. Let $z_{k} \in P_{i}\setminus Q_{i}$ be a point of type $\beta_{k}$ for each $k$. After passing to a further subsequence we may assume that these points are all contained in some $Q_{j}$ for $j\neq i$. However, the sequence $\{z_{k}\}$ must then accumulate onto $x_{j}$. This would then imply that $x_{j} \in P_{i}$ since $P_{i}$ is closed. This contradicts the choice of $P_{i}$. Therefore we conclude that the maximum is realized. 

    The function is $\Hoan$-equivariant since the group acts on $X_{\alpha,n}$ by homeomorphisms.

    Finally, to check the strong triangle inequality we use the triangle inequality for symmetric differences. That is, given three sets $A,B,C$, we have 
    \begin{align*}
        A \triangle B \subset (A \triangle C) \cup (C \triangle B).
    \end{align*}
    Therefore, given three partitions $\cP,\cQ,\cR \in \fP$ we have 
    \begin{align*}
        [P_{i} \triangle Q_{i}]_{\beta} \subset [P_{i} \triangle R_{i}]_{\beta} \cup [R_{i} \triangle Q_{i}]_{\beta}
    \end{align*}
    for all $i$ and $\beta$. Therefore, if both $[P_{i}\triangle R_{i}]_{\beta} = \emptyset$ and $[R_{i}\triangle Q_{i}]_{\beta} = \emptyset$, then $[P_{i}\triangle Q_{i}]_{\beta} = \emptyset$. 
\end{proof}

Next we fix a basepoint partition $\cP \in \fP$. This allows us to define a height function on $\fP$ by setting $h(\cQ) = h(\cP,\cQ)$. For any $\beta < \alpha$ we define the sublevel sets
\begin{align*}
    \fP_{\beta} \defeq \{\cQ \vert h(\cQ) \leq \beta \}.
\end{align*}
We claim that the stabilizers of these sublevel sets form a countable chain of proper open subgroups exhausting $\Hoan$. We break this down into several lemmas. 

\begin{lem}\label{lem:stabopen}
    For each $\beta < \alpha$, $\Stab(\fP_{\beta})$ is a proper open subgroup of $\Hoan$. 
\end{lem}
\begin{proof}
    We first note that $\Stab(\fP_{\beta})$ is a proper subgroup. Indeed, if $g \in \Hoan$ is any homeomorphism mapping a point of type $\beta' > \beta$ from $P_{1}$ to $P_{2}$, then $h(g\cP) \geq \beta'$ and thus $g \notin \Stab(\fP_{\beta})$.

    We next check that $\Stab(\fP_{\beta})$ is an open subgroup. It suffices to see that $\Stab(\fP_{\beta})$ contains an open neighborhood of the identity. We claim that $\Stab(\cP) \subset \Stab(\fP_{\beta})$. Let $g \in \Stab(\cP)$ and $\cQ \in \fP_{\beta}$. By \Cref{lem:heightproperties} we have 
    \begin{align*}
        h(g\cQ) = h(\cP,g\cQ) \leq \sup\{ h(\cP,g\cP), h(g\cP,g\cQ)\} \leq \beta.
    \end{align*}
    The final inequality comes from the fact that $h(\cP,g\cP) = 0$ since $g \in \Stab(\cP)$ and the equivariance of $h$. We conclude that $g \in \Stab(\fP_{\beta})$ and hence $\Stab(\cP)\subset \Stab(\fP_{\beta})$.
\end{proof}

\begin{lem}\label{lem:stabnested}
    If $\delta < \beta < \alpha$, then $\Stab(\fP_{\delta})\lneq\Stab(\fP_{\beta})$. 
\end{lem}
\begin{proof}
    Let $g \in \Stab(\fP_{\delta})$ and $\cQ \in \beta$. Again, by \Cref{lem:heightproperties} we have
    \begin{align*}
        h(g\cQ) = h(\cP,g\cQ) \leq \max\{ h(\cP,g\cP), h(g\cP,g\cQ)\} \leq \beta.
    \end{align*}
    Here we are using the fact that $g \in \Stab(\fP_{\delta})$ implies that $h(\cP,g\cP) \leq \delta$ and the equivariance of $h$. Therefore, $g \in \Stab(\fP_{\beta})$ and $\Stab(\fP_{\delta})<\Stab(\fP_{\beta})$.

    In order to see that $\Stab(\fP_{\delta})$ is a proper subgroup of $\Stab(\fP_{\beta})$ we note that if $g \in \Stab(\fP_{\beta})$ is a homeomorphism that sends a point of type $\beta$ from $P_{1}$ to $P_{2}$ then $g\cP \notin \fP_{\delta}$ and hence $g \notin \Stab(\fP_{\delta})$.
\end{proof}

\begin{lem}\label{lem:stabexhaust}
    We can write $\displaystyle\Hoan = \bigcup_{\beta<\alpha} \Stab(\fP_{\beta})$.
\end{lem}
\begin{proof}
    Let $g \in \Hoan$. By \Cref{lem:heightproperties}, $h$ takes values strictly less than $\alpha$, so there exists some $\beta < \alpha$ so that $h(gP) = \beta$. We claim also that $g \in \Stab(\fP_{\beta})$. Indeed, by \Cref{lem:heightproperties}, for $\cQ \in \fP_{\beta}$ we again have 
    \begin{align*}
        h(g\cQ) = h(\cP,g\cQ) \leq \max\{ h(\cP,g\cP), h(g\cP,g\cQ)\} \leq \beta.
    \end{align*}
    We conclude that $g \in \Stab(\fP_{\beta})$ and hence $\Hoan$ is equal to the countable union $\bigcup_{\beta<\alpha} \Stab(\fP_{\beta})$. 
\end{proof}

These three lemmas taken together provide a proof of \Cref{prop:limitord}. 

\begin{proof}[Proof of \Cref{prop:limitord}]
    The three lemmas above exactly allow one to write $\Hoan$ as the countable union of the proper open subgroups $\{\Stab(\fP_{\beta})\}_{\beta<\alpha}$. Therefore, by \Cref{lem:countgen}, we have that $\Hoan$ does not have a coarsely bounded generating set. 
\end{proof}

\bibliographystyle{alpha}
\bibliography{bib.bib}
\end{document}